\documentclass[12pt]{amsart}
\usepackage{graphicx}
\usepackage{amsfonts}
\usepackage{amsthm}
\usepackage{amsmath}
\usepackage{amssymb}
\usepackage[arrow,matrix,curve,color]{xy}
\usepackage{color}
\usepackage[normalem]{ulem}
\usepackage{tikz}
\usepackage{xspace}
\definecolor{light-blue}{rgb}{0.8,0.85,1}
\definecolor{light-red}{rgb}{1,.4,.4}
\definecolor{purp}{rgb}{.7,.3,1}
\definecolor{yel}{rgb}{1,1,.5}
\definecolor{cy}{rgb}{0,1,1}
\usepackage{colortbl}
\usepackage{multirow}
\usepackage{array}
\newtheorem{theorem}{Theorem}

\newtheorem{proposition}[theorem]{Proposition}
\theoremstyle{definition}

\newtheorem{definition}[theorem]{Definition}

\newtheorem{remark}[theorem]{Remark}

\newcommand{\co}{\colon\,}

\newcommand{\bT}{\mathbb T}
\newcommand{\bR}{\mathbb R}
\newcommand{\bC}{\mathbb C}

\newcommand{\bF}{\mathbb F}

\newcommand{\bZ}{\mathbb Z}
\newcommand{\bO}{\mathbb O}
\newcommand{\bP}{\mathbb P}

\newcommand{\cA}{\mathcal A}
\newcommand{\cAe}{\mathcal A^{\text{even}}}
\newcommand{\cB}{\mathcal B}
\newcommand{\cC}{\mathcal C}

\newcommand{\cH}{\mathcal H}

\newcommand{\SO}{\mathop{\rm SO}}
\newcommand{\SU}{\mathop{\rm SU}}

\newcommand{\Sp}{\mathop{\rm Sp}}
\newcommand{\Spin}{\mathop{\rm Spin}}
\newcommand{\PSU}{\mathop{\rm PSU}}
\newcommand{\PSp}{\mathop{\rm PSp}}

\newcommand{\tH}{\widetilde H}

\newcommand{\pt}{\text{pt}}
\newcommand{\lp}{\textup{(}}
\newcommand{\rp}{\textup{)}}

\newcommand{\Ext}{\operatorname{Ext}}
\newcommand{\Hom}{\operatorname{Hom}}
\newcommand{\Tor}{\operatorname{Tor}}

\newcommand{\Sq}{\operatorname{Sq}}

\newcommand{\lcm}{\operatorname{lcm}}
\newcommand{\rank}{\operatorname{rank}}

\newcommand{\pr}{\text{pr}}
\newcommand{\odd}{\text{\textup{odd}}}

\title[Twisted $K$-theory of compact Lie groups]{A new approach to twisted $K$-theory of compact Lie groups}
\author{Jonathan Rosenberg}
\address{Department of Mathematics\\
University of Maryland\\
College Park, MD 20742-4015, USA} 
\email[Jonathan Rosenberg]{jmr@math.umd.edu}
\thanks{Partially supported by {U.S.} NSF grant number DMS-1607162.
  The author would also like to thank the
  Isaac Newton Institute for Mathematical Sciences, Cambridge, U.K., for
  support and hospitality during the Programme on
  Operator Algebras: subfactors and their applications,
  and the Hausdorff Institute for Mathematics, Bonn, Germany, for
  support and hospitality during the Trimester Program on 
  K-Theory and Related Fields, both in 2017,
  when some work on this paper was undertaken. This work was also
  supported by {U.K.} EPSRC grant number EP/K032208/1.} 
  
\begin{document}
\begin{abstract}
  This paper explores further the computation of the twisted
  $K$-theory and $K$-homology of compact simple Lie groups,
  previously studied by Hopkins, Moore, Maldacena-Moore-Seiberg,
  Braun, and Douglas, with a focus on groups of rank $2$.
  We give a new method of computation based on the
  Segal spectral sequence which seems to us appreciably simpler than
  the methods used previously, at least in many key cases.
\end{abstract}

\maketitle

\section{Introduction}
\label{sec:intro}
This paper is an outgrowth of the paper \cite{MathaiRos} by
Mathai and the author, where we started studying a new approach
to the computation of the twisted $K$-theory of compact simple
Lie groups.  This problem was first studied by physicists
(e.g., \cite{MR2079376,MR1877986,MR1834409,MR1960468,
MR2061550,MR2399311,MR2080884}) because of interest in the
WZW (Wess-Zumino-Witten) model, which 
appears both in conformal field theory and as a string theory 
whose underlying spacetime manifold
is a Lie group, usually compact and simple.  In string theories
in general, D-brane charges are expected to take their values
in twisted $K$-theory of spacetime, so the study of WZW models
led to the study of twisted $K$-theory of compact Lie groups.
The calculation of twisted $K$-theory of Lie groups turned out
to be sufficiently interesting so that it was eventually taken up
by mathematicians (Hopkins, unpublished, but quoted in
\cite{MR1877986}, and Douglas \cite{MR2263220}).

Section \ref{sec:review}
then revisits the topic of computing twisted $K$-theory
$K^\bullet(G,h)$, for arbitrary choices of the twisting $h$.
This has been the subject of an extensive
literature, most notably \cite{MR1877986,MR2079376,
  MR2061550,MR2263220}, and the results are rather complicated
and hard to understand.  However, this is an important problem
because of the connection, discovered by physicists, between
these twisted $K$-groups and fusion rings and representations
of loop groups.  We therefore present in Section \ref{sec:twistedKrank2}
an easier way of computing these twisted $K$-groups for compact
simply connected simple compact Lie groups of rank two.
Theorems \ref{thm:SU3}, \ref{thm:G2oddtorsion},
\ref{thm:G2twotorsion}, and \ref{thm:Sp2} recover all the known
results for rank-$2$ groups using our direct methods.
Section \ref{sec:nonsimplyconn} goes on to discuss
non-simply connected groups. While many of the results of this
paper were previously known, to the best of my knowledge,
Theorems \ref{thm:SegalSSdiff}, \ref{thm:Hur},
\ref{thm:ANSS}, and \ref{thm:SO5} are new.

I would like to thank the referee for a careful reading of the
manuscript, for noticing a mistake, and for making several very
useful suggestions.

\section{Review and Machinery}
\label{sec:review}

In this paper we will deal exclusively with complex periodic $K$-theory,
which is $2$-periodic by Bott periodicity.  Given a
topological space $X$ (which for present purposes we can take to
be compact) and a principal bundle $P$ over $X$ with fibers
the projective unitary group $PU(\cH)=U(\cH)/\bT$
of an infinite-dimensional separable Hilbert
space $\cH$, $P$ defines a bundle of spectra over $X$ with fibers
the $K$-theory spectrum, and from this one can construct
\emph{twisted $K$-theory} of $X$ in a standard way
(see for example
\cite{MR679694,MR1018964,MR2172633,MR2513335,MR2681757,MR2986868}---this
is only a small subset of the literature). 

Since $PU(\cH)\simeq \bC\bP^\infty$ is a $K(\bZ,2)$ space, bundles $P$
as above are classified by classes $h\in H^3(X,\bZ)$, and
we will denote the twisted $K$-theory or twisted $K$-homology of
$X$ by $K^\bullet(X,h)$ or $K_\bullet(X,h)$, even though, strictly
speaking, $h$ only determines these groups up to non-canonical
isomorphism.  (The non-canonicity will not be important in anything we do.)

Let $G$ be a compact Lie group. In this section we restrict
to the case where $G$ is simple, connected, and simply connected,
which is the most studied case.  Since $G$ is then $2$-connected
with $\pi_3(G)\cong \bZ$, we have $H^3(G,\bZ)\cong \bZ$. There
is in fact a canonical isomorphism of $H^3(G,\bZ)$ with $\bZ$
(i.e., a canonical choice of generator),
due to the fact that $(X,Y,Z)\mapsto \langle X, [Y,Z]\rangle$,
$\langle \underline{\phantom{X}},\,  \underline{\phantom{X}}\rangle$
the Killing form, defines a canonical $3$-form on the Lie algebra of
$G$, and thus a preferred orientation on $H^3_{\text{deR}}(G, \bR)$.
In what follows we will mostly consider the case of
twistings $h>0$ (when $H^3(G,\bZ)$ is identified with $\bZ$).
Changing the sign of $h$ preserves the isomorphism
types of $K^\bullet(X,h)$ and $K_\bullet(X,h)$, and when $h=0$,
Hodgkin \cite{MR0214099} proved that $K^\bullet(G)$ is an
exterior algebra over $\bZ$ with $n$ generators, where
$n=\rank G$.  Thus taking $h\ge1$ is no loss of generality.

For $G=\SU(2)=\Sp(1)$, the twisted $K$-theory $K^\bullet(G, h)$
for $h\ne 0$ was already computed in \cite{MR679694},
with the result that it is $0$ in even degree and $\bZ/h$ in
odd degree.  The following result was proved in
\cite{MR2061550,MR2263220}:
\begin{theorem}[{\cite[Theorem 1.1]{MR2263220}}]
\label{thm:twistedKcompact}
For $G$ a simple, connected, and simply connected compact Lie
group, $\rank G=n$,
and for twisting $h>0$, $K_\bullet(X,h)$ {\lp}even
as a ring{\rp} is the tensor product
of an exterior algebra over $\bZ$ on $n-1$ odd-degree
generators with a finite cyclic group of order $c(G, h)$
a divisor of $h$.
\end{theorem}
As we will see, for the cases at least of $\SU(n+1)$, $\Sp(n)$,
and $G_2$, this is not particularly difficult, and the hard
part is to compute the numbers $c(G, h)$.

Incidentally, the distinction between $K_\bullet(X,h)$ and
$K^\bullet(X,h)$ is not particularly important here.  Since
these twisted $K$-groups are the actual $K$-groups of
a continuous trace $C^*$-algebra $A$ over $G$ (having $h$
as Dixmier-Douady class), $K_\bullet(X,h)\cong K^{-\bullet}(A)$ and
$K^{-\bullet}(X,h)\cong K_\bullet(A)$ are related by the universal coefficient
theorem for type I $C^*$-algebras $A$ \cite{MR731763}, which says that
there is a canonical exact sequence
\[
0 \to \Ext^1_\bZ(K_{\bullet+1}(A),\bZ) \to K^{\bullet}(A)
\to \Hom_\bZ(K_\bullet(A), \bZ) \to 0.
\]
Since $A$ here has finitely generated $K$-theory
and $K^\bullet(A)$ is torsion, $K_\bullet(A)$ has
to be torsion, and so $K_\bullet(X,h)$ and
$K^\bullet(X,h)$ agree except for a degree shift.  Thus, for
$\SU(2)$ and $h\ne 0$, $K_\bullet(G, h)$ is $\bZ/h$ in \emph{even} degree
instead of odd degree, and in all other cases (again, with $h\ne 0$),
$K_\bullet(G, h)$ and $K^\bullet(G, h)$ are actually
non-canonically isomorphic.

In \cite{MR2061550}, a simple form for the numbers $c(G, h)$
was proposed, and was proven modulo a conjecture about the
commutative algebra of Verlinde rings.  (The conjecture is that
Verlinde rings are the coordinate rings of complete intersection
affine varieties.) This conjecture
is known for $\SU(n+1)$, $\Sp(n)$, and $G_2$, but to the best of my knowledge
it might still be open for the spin groups and the other exceptional groups
(see, e.g., \cite{MR2228924,MR3037582,MR3614151} for partial results).
Thus the following should be regarded as a definitive theorem
for $\SU(n+1)$, $\Sp(n)$, and $G_2$, but a ``conditional theorem'' in the
other cases.
\begin{theorem}[{\cite{MR2061550}, but note comments above}]
  \label{thm:Braun}
Assume the conjecture in the paragraph above, which is known
at least in types $A_n$, $C_n$, and $G_2$.
For $G$ a simple, connected, and simply connected compact Lie
group, $\rank G=n$, and for twisting $h>0$, the order $c(G, h)$
of the torsion in $K_\bullet(X,h)$ and $K^\bullet(X,h)$ is given
by the formula $c(G, h) = \frac{h}{\gcd(h, y(G))}$,
where the number $y(G)$ is given by the following table:
\begin{center}
\textup{\begin{tabular}{|l|c|}
    \hline
  $G$&$y(G)$ \\ \hline 
  $A_n=\SU(n+1)$ & $\lcm(1, 2, \cdots, n)$\\
  $B_n=\Spin(2n+1)$   & $\lcm(1, 2, \cdots, 2n-1)$\\
  $C_n=\Sp(n)$   & $\lcm(1, 2, \cdots, n, 1, 3, \cdots, 2n-1)$\\
  $D_n=\Spin(2n)$ ($n>3$)  & $\lcm(1, 2, \cdots, 2n-1)$\\
  $G_2$ & 60\\
  $F_4$ & 27720\\
  $E_6$ & 27720\\
  $E_7$ & 12252240\\
  $E_8$ & 2329089562800\\
  \hline
\end{tabular}}
\end{center}
\end{theorem}

Formulas were also given for $c(G, h)$
in \cite[Theorem 1.2]{MR2263220} for the classical groups
and  \cite[p.\ 797]{MR2263220} for $G_2$, but they have a totally
different form; for example,
\[
c(\SU(n+1), h) = \gcd\left(\tbinom{h+i}{i}-1, 1\le i\le n\right)
\]
and
\[
c(\Sp(n), h) =
\gcd\left(\sum_{-h\le j\le -1}\tbinom{2j+2(i-1)}{2(i-1)}, 1\le i\le n\right).
\]

Appendix C in \cite{MR1877986} proved that the Douglas and Braun
formulas coincide in the case of $\SU(n+1)$.  In Propositions
\ref{prop:BraunDougSp2} and 
\ref{prop:BraunDougG2}, we will also see that the Douglas and Braun
formulas coincide in the case of $\Sp(2)$ and $G_2$.

We now move on to the question of how to
prove results like Theorem \ref{thm:twistedKcompact} and
Theorem \ref{thm:Braun} in an easier way. Computation of
$K^\bullet(\SU(n+1),h)$ was discussed in \cite{MR1834409,MR1877986,MR2079376}
using methods motivated by physics, based on a study of wrapping
of branes in WZW theories.  However, those papers don't quite
give a mathematically rigorous proof, except in the simplest
cases.  More sophisticated methods for
computing $K^\bullet(G,h)$ were used in \cite{MR2061550,MR2263220},
but the techniques are decidedly not elementary.
\cite{MR2061550}  used the Hodgkin K\"unneth spectral sequence
in equivariant $K$-theory together with the calculations
of Freed-Hopkins-Teleman \cite{MR1829086,MR2365650}\footnote{There
is indirect physics input here since Freed-Hopkins-Teleman showed
that the \emph{equivariant} twisted $K$-theory is the same
as the Verlinde ring of the associated WZW model.},
while \cite{MR2263220} used a Rothenberg-Steenrod spectral sequence
and $K$-theory of loop spaces.  So our purpose here is
to give a more direct approach.  We will need
the Segal spectral sequence (from 
\cite[Proposition 5.2]{MR0232393}), though for our purposes
it is easiest to reformulate it in homology instead of cohomology.
\begin{theorem}
\label{thm:SegalSS}
Let $F\xrightarrow{\iota} E\xrightarrow{\pr} B$ be a fiber bundle, say of
CW complexes,
and let $h\in H^3(E)$.  Then there is a homological spectral
sequence
\[
H_p(B, K_q(F, \iota^* h)) \Rightarrow K_\bullet(E, h).
\]
\end{theorem}
\begin{proof}
In the absence of the twist, this is precisely the homology dual of the
spectral sequence of \cite[Proposition 5.2]{MR0232393}, in the
case where the cohomology theory used is complex $K$-theory.
If $h=0$, $E=B$ and $F=\pt$, this reduces to the usual
Atiyah-Hirzebruch spectral sequence (AHSS) for $K$-homology.  Similarly,
if $E=B$ and $F=\pt$, but $h\ne 0$, this is the AHSS for twisted
$K$-homology.  To get the general case, we filter
$B$ by its skeleta.  This induces a filtration of $K_\bullet(E, h)$
for which this is the induced spectral sequence (by Segal's proof).
\end{proof}
\begin{remark}
  The spectral sequence of Theorem \ref{thm:SegalSS} will be strongly
  convergent if the ordinary homology of $B$ is bounded.
  This will be the case if $B$ is weakly equivalent to a finite
  dimensional CW complex, and in particular covers all the cases
  considered in this paper.
\end{remark}  
As a simple application of Theorem \ref{thm:SegalSS}, we can immediately
prove the easiest part of Theorem \ref{thm:twistedKcompact}.
(However, this result is rather weak and we will want to improve on it.)
\begin{theorem}
\label{thm:twistedKish-tors}
Let $G$ be a simple, connected, and simply connected compact Lie
group.  For any twisting $h>0$, $K_\bullet(X,h)$ is a finite
abelian group, and all elements have order a divisor of
a power of $h$.  In particular, if $h=1$, then $K_\bullet(X,h)$ vanishes
identically, and if $h=p^r$ is a prime power, then $K_\bullet(X,h)$
is a $p$-primary torsion group.
\end{theorem}
\begin{proof}
First observe that $G$ contains a subgroup
$H\cong \SU(2)\cong \Sp(1)\cong \Spin(3)$
such that the inclusion $H\hookrightarrow G$ is an isomorphism
on $\pi_j$, $j\le 3$.  Assuming this structural fact, the theorem
follows immediately.  Consider the fibration $H\to G\to G/H$.
From Theorem \ref{thm:SegalSS}, we get a spectral sequence
converging to $K_\bullet(X,h)$, with
$E^2_{p,q} = H_p(G/H, K_q(\SU(2), h))$. But $K_q(\SU(2), h)$ is
non-zero only for $q$ even, where it is $\bZ/h$.  Since $E^2$
is thus torsion with all elements of order dividing $h$, the
same is true of $E^\infty$.  And even if there are nontrivial
extensions involved in going from $E^\infty$ to $K_\bullet(X,h)$,
the result still follows.

It still remains to verify the structural statement.  For the
classical groups, $\SU(2)$ sits in $\SU(n)$,
$\Spin(3)$ sits in $\Spin(n)$, and $\Sp(1)$ sits in $\Sp(n)$
for all relevant values of $n$. The fact that these inclusions
are isomorphisms on $\pi_3$ is standard, and follows from the
classical fibrations
\[
\left\{\begin{aligned}
&\SU(n)\to \SU(n+1)\to S^{2n+1},\\
&\Sp(n)\to \Sp(n+1)\to S^{4n+3},\\
&\Spin(n)\to \Spin(n+1)\to S^n,
\end{aligned}\right.    
\]
together with the facts that $\SU(2)$ and $G$
are both $2$-connected.  In the case of $G_2$, there is a fibration
$\SU(2) \to G_2 \to V_{7,2}$ \cite[Lemme 17.1]{MR0064056}.
In the case of $F_4$, there is a fibration
$\Spin(9)\to F_4\to \bO\bP^2$ \cite{MR0034768}.  For the $E$-series
we can use the fibration $F_4\to E_6\to E_6/F_4$ along with what
we know about $F_4$, then use the inclusions $E_6\hookrightarrow E_7
\hookrightarrow E_8$.
\end{proof}

In order to apply Theorem \ref{thm:SegalSS} more precisely,
in some cases we will
need an explicit description of some of the differentials.
Thus the following theorem is useful.  It applies with basically
the same proof to other exceptional homology theories, though
we won't need these here.
\begin{theorem}
\label{thm:SegalSSdiff}
In the situation of Theorem \textup{\ref{thm:SegalSS}},
suppose that $\iota^*$ is an isomorphism
{\lp}or even just an injection{\rp} on $H^3$
{\lp}so that the twisting class on $E$ can be identified
with the restricted twisting class on $F${\rp}, 
the differentials $d^2,\cdots,d^{r-1}$
leave $E^2_{r,0}=H_r(B, K_0(F, \iota^*h))$ unchanged, and one
has a class $x$ in this group which comes from a class
$\alpha \in \pi_r(B)$ under the composite
\[
\pi_r(B)\xrightarrow{\text{Hurewicz}} 
H_r(B, K_0(F, \iota^*h)).
\]
Then $d^r(x) \in E^r_{0,r-1}$, a quotient of $K_{r-1}(F, \iota^*h)$,
is the image of $\alpha$ under the composite
\[
\pi_r(B)\xrightarrow{\partial} \pi_{r-1}(F)
\xrightarrow{\text{Hurewicz}} K_{r-1}(F, \iota^*h),
\]
where the first map is the boundary map in the long exact sequence
of the fibration $F\xrightarrow{\iota} E\xrightarrow{\pr} B$.
{\lp}The Hurewicz map in twisted homology is easy to understand
as follows, at least if $r\ge 5$: $\iota^*h$ defines a principal
$K(\bZ,2)$-bundle $P_{\iota^*h}$ over $F$, and the pull-back of this
bundle to the total space $P_{\iota^*h}$ is trivial, so there is a natural map
$K_\bullet(P_{\iota^*h})\to K_\bullet(F, \iota^*h)$
\textup{\cite[p.\ 536]{MR2832567}}; the Hurewicz map is
the composite
\[
\pi_{r-1}(F)\cong \pi_{r-1}(P_{\iota^*h}) \xrightarrow{\wedge 1}
\pi_{r-1}(P_{\iota^*h} \wedge \textbf{K}) = K_{r-1}(P_{\iota^*h})
\to K_{r-1}(F, \iota^*h),
\]
where $\pi_{r-1}(F)\cong \pi_{r-1}(P_{\iota^*h})$ if $r>4$, by the
long exact sequence of the fibration
$K(\bZ,2)\to P_{\iota^*h} \to F$.{\rp}
\end{theorem}
\begin{proof}
Since the class $x$ by assumption was not changed under the
earlier differentials, and since the twisting comes entirely
from the fiber, we can, without loss of generality,
reduce to the case where $B$ is a sphere $S^r$ and thus
$E = (\bR^r\times F)\cup F$, where $\bR^r\times F$ corresponds to
$\pr^{-1}$ of the open $r$-cell in the base.  In this case the
spectral sequence comes directly from the long exact sequence
\begin{multline}
\label{eq:LESfibbasesphere}
\cdots \to K_r(F, \iota^*h) \xrightarrow{\iota_*} K_r(E, h) \to K_r(E, F, h)\\
\cong K_r(E\smallsetminus F, h)\cong 
K_{0} (F, \iota^*h) \xrightarrow{\partial}
K_{r-1}(F, \iota^*h) \to \cdots.
\end{multline}
Note here that $H_r(B, K_0(F, \iota^*h))$ can be identified with
the term \linebreak
$K_{0} (F, \iota^*h)$ in \eqref{eq:LESfibbasesphere}.  So the
differential $d^r$ is the boundary map in \eqref{eq:LESfibbasesphere},
and we use commutativity of the diagram
\[
\xymatrix{
  \pi_r(B) \ar[r]^\partial \ar[d]^{\text{Hurewicz}}
  & \pi_{r-1}(F)\ar[d]^{\text{Hurewicz}}\\
  H_r(B, K_0(F, \iota^*h)) \ar[r]^(.55){\partial} & K_{r-1}(F, \iota^*h),
}
\]
a consequence of naturality of the Hurewicz homomorphism.
\end{proof}  

Another useful result for us will be the ``universal coefficient
theorem'' of Khorami \cite{MR2832567}.
\begin{theorem}[Khorami \cite{MR2832567}]
\label{thm:Khorami}
  Let $X$ be a space {\lp}say, a compact CW-complex{\rp}, let
  $h\in H^3(X, h)$, and let $P_h$ be the associated principal
  bundle with structure group $PU(\cH)\simeq \bC\bP^\infty$.
  Then $K_\bullet(X, h) \cong K_\bullet(P_h)\otimes_R \bZ$,
  where $R = K_0(\bC\bP^\infty)$ is a ring under Pontrjagin product
  acting on $K_\bullet(P_h)$ via the principal $\bC\bP^\infty$-bundle
  structure on $P_h$ and on $\bZ$ via the ring homomorphism
  $R\to \bZ$ sending $\beta_j\mapsto 1$, $j=0\text{ or }1$,
  $\beta_j\mapsto 0$, $j>1$.  Here $R$ is the free $\bZ$-module on
  generators $1=\beta_0,\beta_1, \cdots$, where
  $\beta_j$ is dual to $(\gamma-1)^j$,
  $\gamma$ the Hopf line bundle in $K(\bC\bP^\infty)$.
\end{theorem}  
In fact, Khorami mentions
at the end of his paper that he suspects that his theorem can be
used to recover Theorem \ref{thm:twistedKcompact}, though
he gives no details except in the case $G=\SU(2)$, where he
points out that for $P_h$ as in Theorem \ref{thm:Khorami},
$K_\bullet(P_h)\cong R/(h\beta_1)$ and thus
$K_\bullet(\SU(2), h) \cong R/(h\beta_1) \otimes_R \bZ \cong \bZ/h$.

\section{Twisted $K$-theory of rank-two simple Lie groups}
\label{sec:twistedKrank2}

\subsection{The case of $\SU(3)$}
\label{sec:SU3}
To explain how we use these tools, we start with the simplest
nontrivial case, namely $G=\SU(3)$, which was first treated in
\cite{MR1877986,MR2079376}. We recall the result:
\begin{theorem}
\label{thm:SU3}
Let $h$ be a positive integer, viewed as a twisting class for
$\SU(3)$. Then {\lp}in both even and odd degree{\rp},
$K_\bullet(\SU(3), h)\cong \bZ/h$ if $h$ is odd,
$K_\bullet(\SU(3), h)\cong \bZ/(h/2)$ if $h$ is even.
\end{theorem}
\begin{proof}
We use the standard fibration
\[
\SU(2)=S^3 \xrightarrow{\iota} \SU(3)\xrightarrow{\pr}  S^5.
\]
Here $\iota^*$ is an isomorphism on $H^3$, and we already
know that $K_\bullet(\SU(2), h)$ is $\bZ/h$ in even degree, $0$ in
odd degree. So apply Theorem \ref{thm:SegalSS} and
Theorem \ref{thm:SegalSSdiff}.  The picture of the spectral
sequence is given in Figure \ref{fig:SSSU3}.
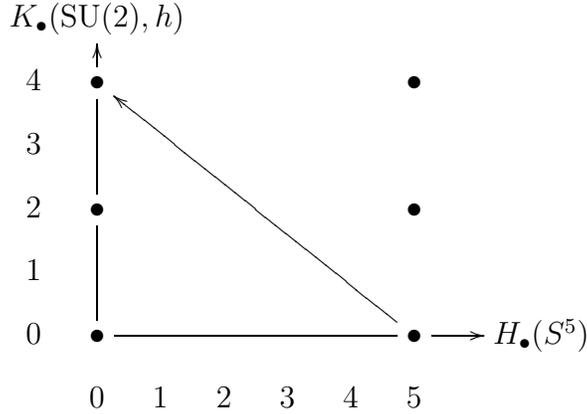
\begin{figure}[hbt]
\[
\xymatrix@!0{
&K_\bullet(\SU(2),h)&&&&&&&\\
4&\bullet\ar[u]&&&&&\bullet&&\\
3&&&&&&&&\\
2&\bullet\ar@{-}[uu]&&&&&\bullet&&\\
1&&&&&&&&&\\
0&\bullet\ar@{-}[rrrrr] \ar@{-}[uu] &&&&
&\bullet\ar[rr]\ar@[red][llllluuuu]&& H_\bullet(S^5)\\
&0&1&2&3&4&5&&
}\]
\caption{The Segal SS for twisted $K$-homology of $\SU(3)$.
  Heavy dots indicate copies of $\bZ/h$. The diagonal
  arrow shows the differential $d^5$.}
\label{fig:SSSU3}
\end{figure}
To compute the differential $d^5$, we use Theorem \ref{thm:SegalSSdiff},
along with the exact homotopy sequence
\[
\pi_5(\SU(2))\to \pi_5(\SU(3))\to \pi_5(S^5)
\xrightarrow{\partial} \pi_4(\SU(2))\to \pi_4(\SU(3)).
\]
Here it is classical that $\pi_5(\SU(2))\cong \pi_4(\SU(2))\cong \bZ/2$,
and $\pi_5(\SU(3))\cong \bZ$, $\pi_4(\SU(3))=0$ by \cite{MR0169242}.
Thus the boundary map $\partial$ in this sequence has kernel of index $2$.
Now we need to understand the Hurewicz maps
\[
\pi_5(S^5)\to H_5(S^5, K_0(\SU(2), h)),
\]
\[
\pi_4(\SU(2))\to K_4(\SU(2), h)\cong K_0(\SU(2), h).
\]
The generator of $\pi_5(S^5)$ is suspended
from the generator of $\pi_0(S^0)$, just as
$H_5(S^5, K_0(\SU(2), h))\cong K_0(\SU(2), h)$ via suspension, so the
generator $1$ of $\pi_5(S^5)$ goes to the generator $1$ of the cyclic
group $\bZ/h$.  To finish the proof, we need the
following Theorem \ref{thm:Hur}.
Thus we see that $d^5$ in the Segal spectral sequence has kernel
of order $2$ if $h$ is even and trivial kernel if $h$ is odd,
and the result follows.
\end{proof}
\begin{theorem}
Let $h\in \bZ$, $h\ne 0$.
Then the Hurewicz map $\pi_4(S^3)\to K_4(S^3, h)\cong \bZ/h$ is
non-zero if and only if $h$ is even.
\label{thm:Hur}
\end{theorem}
\begin{proof}
Since $\pi_4(S^3)\cong \bZ/2$, obviously the Hurewicz map is $0$ if
$h$ is odd, since then there is no $2$-torsion in $K_4(S^3, h)$.
So assume $h$ is even.  The Hurewicz map in twisted $K$-homology
is a bit more mysterious than the usual Hurewicz map in $K$-homology,
but we can apply Theorem \ref{thm:Khorami} to help clarify things. 
Let $P_h$ be the principal $\bC\bP^\infty$-bundle over $S^3$
classified by the non-zero integer $h\in \bZ\cong H^3(S^3, \bZ)$.
The Serre spectral sequence for the fibration
$\bC\bP^\infty \to P_h\to S^3$ has only two columns, so in cohomology
the only differential is $d_3$, which sends the generator $u$ of
$H^2(\bC\bP^\infty)$ to $h$ times the usual generator $y$ of $H^3(S^3)$.
Since $d_3$ is a derivation, $d_3(u^n) = nhu^{n-1}y$, and the
homology differential is similar, but just points in the opposite
direction.  Hence $H_{2n}(P_h) \cong \bZ/(nh)$ for $n\ge 1$,
and $H_{\text{odd}}(P_h)$ vanishes. Thus the AHSS for $P_h$ collapses, and
Khorami computed that $K_0(P_h)\cong R/(h\beta_1)$ as an $R$-module,
where $R = K_0(\bC\bP^\infty)$ (with multiplication defined
by Pontrjagin product), and $K_0(S^3, h)\cong \bZ/h$ is
gotten from this by tensoring with
$\bZ$ (viewed as an $R$-module under $\beta_1\mapsto 1$,
$\beta_j\mapsto 0$, $j>1$).

Note that a map $S^3 \xrightarrow{h} S^3$ of degree $h$ pulls the
$\bC\bP^\infty$-bundle $P_1$ over $S^3$ back to the
$\bC\bP^\infty$-bundle $P_h$ over $S^3$.  So we get a pull-back square
\begin{equation}
  \label{eq:Ph}
  \xymatrix{{\bC\bP^\infty}\ar[r]\ar[d]^= &
    P_h \ar[r]\ar[d]^{f_h} & S^3\ar[d]^h\\
    {\bC\bP^\infty}\ar[r] & P_1\ar[r] & S^3,}
\end{equation}
and the map $f_h\co P_h\to P_1$ induces a map of $R$-modules
$R/(h\beta_1)\to R/(\beta_1)$ on $K$-homology. Comparison
of the Serre spectral sequences also shows that $(f_h)_*$ is
surjective on integral homology.  From the long
exact homotopy sequences associated to the two rows, we also see
that $\pi_j(P_h)\cong \bZ/h$ for $j=2$ and $\cong \pi_j(S^3)$
for $j\ge 4$, and that the map $f_h\co P_h\to P_1$ induces multiplication
by $h$ on $\pi_j$ for $j\ge 4$.

Now note that $P_1$ is the homotopy fiber of the canonical map
$S^3\to K(\bZ,3)$ inducing an isomorphism on $\pi_3$, and
thus $P_1\to S^3\to K(\bZ,3)$ is the beginning of the Postnikov tower of
$S^3$.  Thus $P_1$ is $3$-connected and $\pi_j(P_1)\cong \pi_j(S^3)$
for $j\ge 4$.  So by the Hurewicz theorem, the Hurewicz map
$\pi_4(P_1)\to H_4(P_1)\cong \bZ/2$ is an isomorphism.

We can also consider the diagram
\[
\xymatrix{&P_1 \ar[d]\ar@{.>}[ld]_(.45){i_h}& \\
  P_h \ar[r] &S^3 \ar[r]^(.35)h & K(\bZ, 3),}
\]
where the downward solid arrow is the bundle projection of
the ${\bC\bP^\infty}$-bundle $P_1$ over $S^3$.  From the
exact homotopy sequence
\[
  [P_1, P_h] \to [P_1, S^3] \to [P_1, K(\bZ,3)],
\]
we see that we get a lifting $i_h\co P_1\to P_h$, which is the
first stage of the Postnikov fibration
$P_1\xrightarrow{i_h} P_h \to K(\bZ/h, 2)$ for $P_h$.
Unlike the map $f_h$ in the other direction, $i_h$ is not a map
of ${\bC\bP^\infty}$-bundles.

Putting everything together,
we see that the Hurewicz map $\pi_j(S^3)\to K_j(S^3, h)$
(for $j\ge 4$ even) is the composite
\[
\pi_j(S^3)\cong \pi_j(P_1)\xrightarrow[\cong]{(i_h)_*} \pi_j(P_h)
\to K_j(P_h) = R/(h\beta_1) \to \bZ/h.
\]
Now $K_{\text{even}}(P_1)$ and $K_{\text{even}}(P_h)$ have skeletal
filtrations $F_0=\bZ\subset F_1\subset F_2\subset \cdots$,
where $F_j$ is generated (additively) by the images of
$\beta_0,\cdots,\beta_j$, and since the AHSS for $K$-homology of
$P_1$ collapses,
we have maps $F_j\to H_{2j}$ identifying $F_j/F_{j-1}$ with $H_{2j}$.
The image of $\pi_4$ under the Hurewicz map must lie in $F_2$
(just on dimensional grounds). Thus since
the Hurewicz map $\pi_4(P_1)\to H_4(P_1)$ is an isomorphism,
the Hurewicz map in $K$-homology for $P_1$ maps $\pi_4(P_1)\cong \bZ/2$
onto the cyclic group generated by $\beta_2$, of order $2$ in $R/(\beta_1)$,
that maps onto $H_4(P_1)$. (One can compute that
$2\beta_2 = \beta_1^2-\beta_1$ lies in the ideal generated by $\beta_1$.)

The Hurewicz map in ordinary homology $\pi_4(P_h)\to H_4(P_h)$
can be identified with the edge homomorphism
$(i_h)_*\co H_4(P_1) \to H_4(P_h)$ associated to the Serre spectral
sequence for $P_1\xrightarrow{i_h} P_h \to K(\bZ/h, 2)$,
which is shown schematically in Figure \ref{fig:SSPost}.
\begin{figure}[hbt]
\[
\xymatrix@!=1em{
&H_\bullet(P_1)&&&&&&&\\
4&\bullet\ar[u]&&\bullet&\bullet&\bullet&\bullet&&\\
3&&&&&&&&\\
2&&&&&&&&\\
1&&&&&&&&&\\
0&\bullet\ar@{-}[rr] \ar@{-}[uuuu] &&\bullet\ar@{-}[rr]&&\bullet\ar@{-}[r]
&\bullet\ar[rr]\ar@[red][llllluuuu]^{\color{red}d^5}&& H_\bullet(K(\bZ/h, 2))\\
&0&1&2&3&4&5&&
}\]
\caption{The Serre SS for the first Postnikov fibration of
  $P_h$.}
\label{fig:SSPost}
\end{figure}
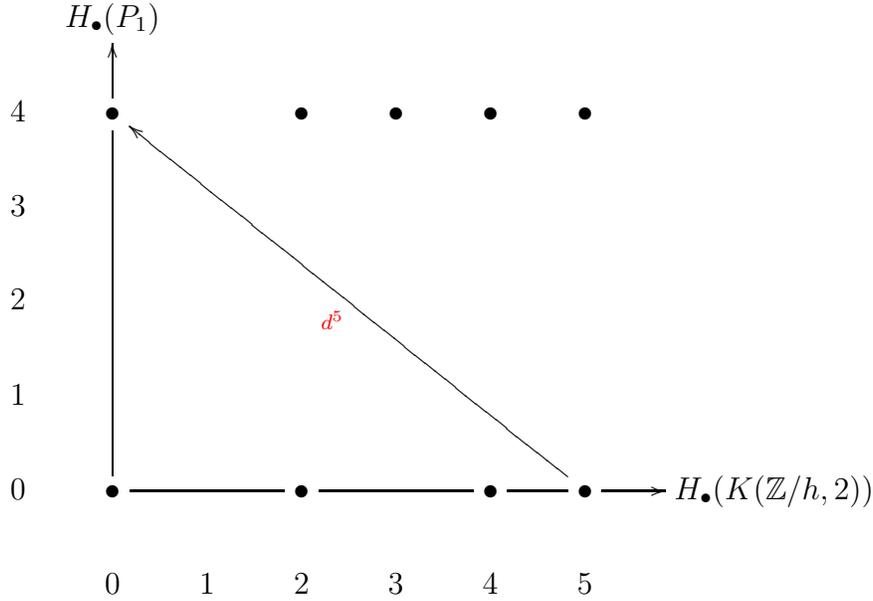
The integral homology of $K(\bZ/h, 2)$ is a bit complicated,
but we only need its $2$-primary part in low degree. When
$h=2$, Serre showed that $H^\bullet(K(\bZ/h, 2),\bF_2)$ is a polynomial
ring on generators $\iota, \Sq^1\iota, \Sq^2\Sq^1\iota$, $\cdots$,
where $\iota$ is the canonical generator in degree $2$
\cite[p.\ 500]{MR1867354}.  Thus the $\bF_2$-Betti numbers of
$K(\bZ/2, 2)$ are $1,0,1,1,1,2,\cdots$.  In particular
we can see from this that $\rank H_4(K(\bZ/2, 2); \bZ)=1$.
The complete calculation of $H_4(K(\bZ/h, 2); \bZ)$ may be found
in \cite[Theorem 21.1]{MR0065162} and in \cite{MR0065161,Cartan} (which even
computes the integral homology in arbitrary degree, at least in principle),
and it turns out that $H_4(K(\bZ/h, 2); \bZ)\cong \Gamma_4(\bZ/h)$,
where $\Gamma_4$ is the functor defined in \cite{MR0035997},
and for $h$ even, this is a cyclic group of order $2h$, while
$H_5(K(\bZ/h, 2); \bZ)\cong \bZ/2$ \cite[Theorem 22.1]{MR0065162}.

But recall that $H_4(P_h; \bZ)\cong \bZ/(2h)$.
Thus the red arrow in Figure \ref{fig:SSPost}
has to be an isomorphism and the edge homomorphism
$(i_h)_*\co H_4(P_1)\to H_4(P_h)$, which is the Hurewicz map, vanishes.
One way of thinking about this is that we can view the Hurewicz map
as being about the embedding of an $S^4$ in $P_h$ via the generator
$\eta$ of $\pi_4(P_h)$. This sphere doesn't bound a disk (if it did,
the homotopy class of $\eta$ would be trivial), but it does bound a
homology chain, and even an oriented manifold (in this low dimension,
oriented bordism is almost the same as homology).
The question for us now is: does it bound
a $\text{Spin}^c$ manifold?  This determines the Hurewicz map
in $K$-homology since (when we localize everything at $2$)
the low-degree summand of $M\text{Spin}^c$ is
$ku$ (connective $K$-theory) (\cite[\S8]{MR0234475}
and \cite[Ch.\ XI]{MR0248858}) and the $K$-homology class of this
$4$-sphere, which is the image of the Hurewicz map, comes from its
$ku$-homology class.

So let's reconsider Figure \ref{fig:SSPost} redone in $ku$-homology,
which is  Figure \ref{fig:kuPost}.
\begin{figure}[hbt]
\[
\xymatrix@!=1em{
&ku_\bullet(P_1)&&&&&&&\\
4&\bullet\ar[u]&&\bullet&\bullet&\bullet&\bullet&&\\
3&&&&&&&&\\
2&\bullet&&\bullet&&\bullet&\bullet\ar[llluu]^{d^3}&&\\
1&&&&&&&&&\\
0&\bullet\ar@{-}[rr] \ar@{-}[uuuu] &&\bullet\ar@{-}[rr]&&\bullet\ar@{-}[r]
&\bullet\ar[rr]\ar@[red][llllluuuu]_{\color{red}d^5}
\ar@[blue][llluu]^{\color{blue}d^3}&& H_\bullet(K(\bZ/h, 2))\\
&0&1&2&3&4&5&&
}\]
\caption{The Segal SS in $ku$ for the first Postnikov fibration of
  $P_h$.}
\label{fig:kuPost}
\end{figure}
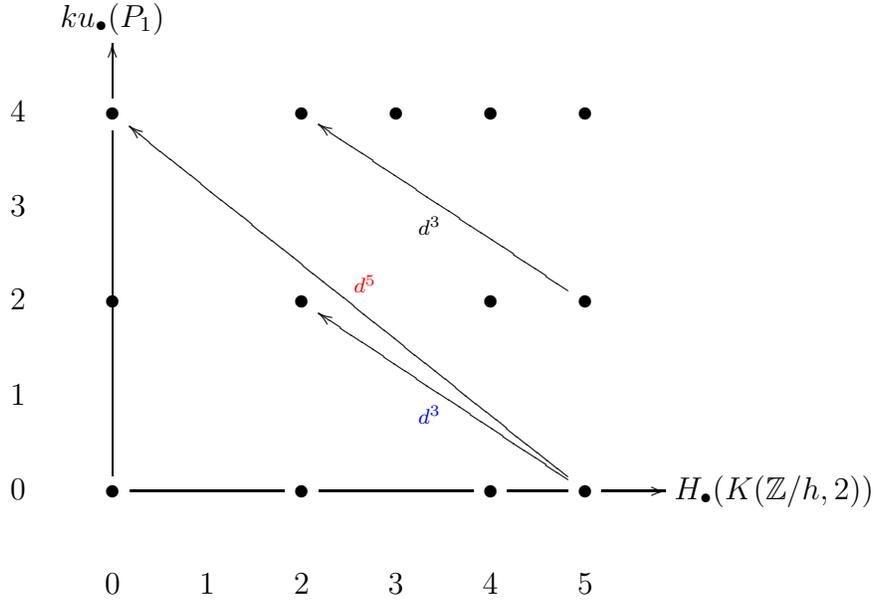

Since $ku_\bullet(P_h)$ is all concentrated in even degree,
everything in odd degree must cancel.  We have
$\widetilde{ku}_2(P_1)=0$,
$\widetilde{ku}_4(P_1) \cong H_4(P_1)\cong \bZ/2$,
$\widetilde{ku}_2(P_h) \cong H_2(P_h) \cong \bZ/h$,
and $\widetilde{ku}_4(P_h)$ is an extension of
$H_4(P_h) \cong \bZ/(2h)$ by $H_2(P_h) \cong \bZ/h$.
The Hurewicz map $\pi_4(P_h)\to K_{\text{even}}(P_h)$ now comes
from the edge homomorphism $\widetilde{ku}_4(P_1) \to
\widetilde{ku}_4(P_h)$, and so the relevant question is which
of the two arrows ($d^3$ and $d^5$) starting at the position
$(5,0)$ in Figure \ref{fig:kuPost} is non-zero.

To answer this question we can consider the map
$P_h\to K(\bZ/h, 2)$, which induces a morphism of spectral
sequences from the spectral sequence of Figure \ref{fig:kuPost}
to the Atiyah-Hirzebruch spectral sequence for computing
$ku_\bullet(K(\bZ/h, 2))$. By \cite{MR0231369} and
\cite[Theorem 2]{MR0339131}, $\widetilde K_\bullet(K(\bZ/h, 2))$ vanishes.
Looking then at the AHSS for $ku_\bullet(K(\bZ/h, 2))$, we see that
$d^3\co H_5\to H_2$ is non-zero, since the dual differential
for computing $ku$-cohomology,
\[
d_3=\Sq^3\co H^3(K(\bZ/h, 2);\bZ)\to H^6(K(\bZ/h, 2);\bZ)
\]
is non-zero, and thus the blue arrow in
Figure \ref{fig:kuPost} is non-trivial.  This implies that the map
$\widetilde{ku}_4(P_1)\to \widetilde{ku}_4(P_h)$ is non-trivial,
and the image must go to $\left(\frac{h}{2}\right)\beta_1$
in $\widetilde K_{\text{even}}(P_h)$. So under the
map $K_0(P_h)\to K_0(S^3, h)$, it maps to $\frac{h}{2}$
in $\bZ/h$, which is non-trivial.  This gives the desired result.
\end{proof}

\subsection{The Braun-Douglas Theorem for $\Sp(2)\cong \Spin(5)$}
\label{sec:Sp2}
We begin by showing that the
Douglas and Braun formulas coincide in the case of $\Sp(2)\cong \Spin(5)$.
Here it is convenient to use the following standard definition.
\begin{definition}
\label{def:porder}
Fix a prime $p$, and for $x$ a positive integer, let
$\nu_p(x)$ be the number of times that $p$ divides $x$. In other
words, $\nu_p(x)$ is defined by the property that $x=p^{\nu_p(x)}x'$,
where $\gcd(x', p)=1$. Thus $x=\prod_{p\text{ prime}}p^{\nu_p(x)}$.
\end{definition}
\begin{proposition}
\label{prop:BraunDougSp2}
Let $h$ be a positive integer, let
\[
h' = \frac{h}{\gcd(h, \lcm(1,2,3))}= \frac{h}{\gcd(h, 6)},
\]
which is the Braun formula for $c(G_2, h)$, and let
\[
h'' = \gcd\left(h, 2\tbinom{h}{3} + \tbinom{h}{2} \right),
\]
which is Douglas' formula for $c(G_2, h)$ in
\cite[Theorem 1.2]{MR2263220}. Then $h'=h''$.
\end{proposition}
\begin{proof}
It will suffice to show that $\nu_p(h')=\nu_p(h'')$ for all primes $p$.
Clearly $\nu_p(h')=\nu_p(h'')=0$ if $p$ does not divide $h$.  So
assume $p$ divides $h$, and we'll consider in turn the cases of
$p=2,3$, and $p>3$.  If $p=2$ and $h$ is even, then
$\nu_2(h')=\nu_2(h)-1$,
while $2\tbinom{h}{3} = \frac{1}{3}h(h-1)(h-2)$ is divisible by $8$
and $\tbinom{h}{2} = \frac{h}{2}(h-1)$ has the same divisibility by $2$
as $\frac{h}{2}$. Thus $\nu_2(h'')=\nu_2(h)-1 = \nu_2(h')$. If
$p=3$ and $h\equiv0\pmod3$, then $\nu_3(h')=\nu_3(h)-1$, while
$\nu_3\left(\tbinom{h}{2}\right) = \nu_3(h)$ and
\[
\nu_3\left(2\tbinom{h}{3}\right) = \nu_3\left(\frac{h(h-1)(h-2)}{3}\right)
=\nu_3(h)-1.
\]
Thus $\nu_3(h'')=\nu_3(h)-1$. Finally, if $p>3$ and $p$ divides $h$,
then $\nu_p(h) = \nu_p\left(\tbinom{h}{2}\right) =
\nu_p\left(2\tbinom{h}{3}\right)$ and so
$\nu_p(h')=\nu_p(h'')=\nu_p(h)$.
\end{proof}
\begin{theorem}[Braun-Douglas Theorem for $\Sp(2)$]
\label{thm:Sp2}
Let $h$ be a positive integer and let $G=\Sp(2)\cong \Spin(5)$.
Then in any degree,
$K_\bullet(G,h)$ is cyclic of order $\frac{h}{\gcd(h,6)}$.  
\end{theorem}
\begin{proof}
We argue as in Theorem \ref{thm:SU3},
using the usual fibration $\Sp(1)\to G\to S^7$, where
$\Sp(1)\cong \SU(2)\cong S^3$. The inclusion of $\Sp(1)$ into $\Sp(2)$
induces an isomorphism on $H^3$, and the groups $K_j(\Sp(1), h)$
are cyclic of order $h$ for $j$ even, zero for $j$ odd.
We just need to compute the differential in the Segal spectral
sequence
\[
d^7\co H_7(S^7, K_0(\Sp(1), h)) \to K_6(\Sp(1), h).
\]
As explained in Theorem \ref{eq:LESfibbasesphere}, this differential
is related to the boundary map $\partial$ in the long exact homotopy
sequence
\[
\pi_7(\Sp(1)) \to \pi_7(\Sp(2)) \to \pi_7(S^7)
\xrightarrow{\partial} \pi_6(\Sp(1)) \to \pi_6(\Sp(2)).
\]
Here $\pi_6(S^3)\cong \bZ/12$, $\pi_7(S^3)\cong \bZ/2$,
$\pi_7(\Sp(2))\cong \bZ$, and $\pi_6(\Sp(2))=0$.
(See for example \cite[p.\ 339]{MR1867354} for the homotopy
groups of $S^3$ and \cite{MR0169242} for the homotopy
groups of $\Sp(2)$.)  From this exact sequence,
$\partial$ is  surjective onto $\pi_6(S^3)\cong \bZ/12$.
Away from the primes $2$ and $3$, $\partial$ vanishes
and so does the differential in the spectral sequence for
$K_\bullet(G,h)$. So we only need to analyze what happens at the
primes $2$ and $3$.  This involves understanding the
Hurewicz homomorphism $\pi_6(S^3)\to K_6(S^3, h)$
or $\pi_6(P_h)\to K_6(P_h)$, where $P_h$ is the principal
$\bC\bP^\infty$-bundle over $\SU(2)$ associated to the twist $h$
as in the proof of Theorem \ref{thm:Hur}.

We can proceed as we did there.  If $h$ is divisible by neither
$2$ nor $3$, then obviously the Hurewicz homomorphism is zero.
If $h$ is divisible by $3$ and we localize at $3$, then everything
is largely as in the proof of Theorem \ref{thm:Hur}, and we keep the notation
used there.  The fiber $P_1$ of the Postnikov fibration
$P_1\xrightarrow{h} P_h \to K(\bZ/h, 2)$ is $3$-locally $5$-connected,
so by the mod-$\cC$ Hurewicz Theorem (with $\cC$ the Serre class of
prime-to-$3$ torsion groups), the $3$-torsion subgroup $\bZ/3$ of
$\pi_6(S^3)\cong \pi_6(P_1)\cong \pi_6(P_h)$ maps isomorphically to 
$H_6(P_1)\cong \bZ/3$. We have $H_6(P_h)\cong \bZ/(3h)$ and
$H_6(K(\bZ/h, 2))\cong \bZ/(3h)$ by \cite[Theorem 21.1]{MR0065162},
so the differential $d^7\co H_7(K(\bZ/h, 2))\to H_6(P_1)$
in the analogue of Figure \ref{fig:SSPost} must be non-zero and the
Hurewicz homomorphism in ordinary homology, which can be identified
with the map $(i_h)_*\co H_6(P_1)\to H_6(P_h)$, vanishes.
To study the corresponding map in $ku$-homology, 
we compare the diagram analogous to Figure \ref{fig:kuPost} with the
AHSS for $ku_\bullet(K(\bZ/h, 2))$.  $H^\bullet(K(\bZ/h, 2);\bF_p)$
has generators $\iota_2, \,\beta_r\iota_2,\,P^1\beta_r\iota_2$, etc.\
($\beta_r$ the $r$th-power Bockstein, $3^r$ the biggest power of $3$
dividing $h$) and the first nontrivial differential in the AHSS
for computing $K^\bullet(K(\bZ/3^r, 2))$ from $H^\bullet(K(\bZ/3^r, 2);\bZ)$
is (up to a non-zero constant)
$d_5= \beta_rP^1\co \beta_r\iota_2\mapsto \beta_rP^1\beta_r\iota_2$.
This is dual to a non-zero differential $d^5$ with target
$H_7(K(\bZ/3^r, 2);\bZ)$, and so the Hurewicz map $ku_6(P_1)\to ku_6(P_h)$
will be non-zero, just as in the proof of Theorem \ref{thm:Hur}.

The hardest step is the $2$-local calculation in the case where $h$
is even, which involves the $2$-local part of the Hurewicz map
$\pi_6(P_h)\to ku_6(P_h)$ for $h$ even. We defer this calculation
to Theorem \ref{thm:ANSS}.
\end{proof}

\subsection{The Braun-Douglas Theorem for $G_2$}
\label{sec:G2}
The following result and its proof are partially modeled on
Appendix C in \cite{MR1877986}, and proves that the
Douglas and Braun formulas coincide in the case of $G_2$.
\begin{proposition}
\label{prop:BraunDougG2}
Let $h$ be a positive integer, let
\[
h' = \frac{h}{\gcd(h, 60)},
\]
and let
\[
h'' = \gcd\left(h, \tbinom{h+2}{2}-1,
\frac{(h + 1)(h + 2)(2h + 3)(3h + 4)(3h + 5)}{120} - 1 \right).
\]
Note that $h'$ is Braun's formula for $c(G_2, h)$
and $h''$ is Douglas' formula for $c(G_2, h)$. Then $h'=h''$.
\end{proposition}
\begin{proof}
We again use Definition \ref{def:porder}.
It will suffice to show that $\nu_p(h')=\nu_p(h'')$ for all primes $p$.

First consider $p=2$. If $h$ is odd, then $\nu_2(h)=\nu_2(h')=\nu_2(h'')=0$. 
If $\nu_2(h)=1$, then since $\nu_2\left(\gcd(h,60)\right)=1$,
$\nu_2(h')=0$.  Consider $h''$.  We have $h\equiv2\pmod4$, so
$h+2\equiv0\pmod4$ and $\binom{h+2}{2}$ is even, hence
$\binom{h+2}{2}-1$ is odd.  Thus $\nu_2(h'')=0=\nu_2(h')$ in this case.
If $\nu_2(h)\ge2$, then since $\nu_2(60)=2$, $\nu_2(h')=\nu_2(h)-2$.
But $\tbinom{h+2}{2}-1 = \frac{(h+2)(h+1)-2}{2}=\frac{h(h+3)}{2}$.
Thus if $\nu_2(h)\ge2$, $\nu_2(\gcd(h, \binom{h+2}{2}-1)) =
\nu_2(h)-1$.  On the other hand
\[
\begin{aligned}
&\frac{(h + 1)(h + 2)(2h + 3)(3h + 4)(3h + 5)}{120} - 1\\
&= \frac{18h^5 + 135h^4 + 400h^3 + 585h^2 + 422h + 120 - 120}{120}\\
&= \frac{h(18h^4 + 135h^3 + 400h^2 + 585h + 422)}{120}.
\end{aligned}
\]
The denominator is $2^3\cdot 15$ and since $h\equiv 0\pmod4$
and $422\equiv 2\pmod4$, $\nu_2$ of the numerator is $\nu_2(h)+1$.
Thus $\nu_2$ of this fraction is $\nu_2(h)+1-3 = \nu_2(h)-2 = \nu_2(h')$.
So again $\nu_2(h') = \nu_2(h'')$.

Next, consider $p=3$. If $\nu_3(h)=0$, then clearly
$\nu_3(h')=\nu_3(h'')=0$.  If $\nu_3(h)\ge 1$, then
$\nu_3(\gcd(h, 60)) = 1$, so $\nu_3(h')=\nu_3(h)- 1$.
On the other hand,
$\nu_3\left(\frac{h(h+3)}{2}\right)> \nu_3(h)$, so taking
the gcd with $\frac{h(h+3)}{2}$ doesn't change $\nu_3(h)$.
With regard to
$\frac{h(18h^4 + 135h^3 + 400h^2 + 585h + 422)}{120}$, if $h$ is
divisible by $3$, then $\nu_3$ of the numerator is the same
as for $h$ (since $422\equiv2\pmod3$), while $\nu_3(120)=1$, so
$\nu_3$ of the fraction, as well as $\nu_3(h'')$, is $\nu_3(h)-1$,
which agrees with $\nu_3(h')$.

Consider now $p=5$. If $\nu_5(h)=0$, then clearly
$\nu_5(h')=\nu_5(h'')=0$.  If $\nu_5(h)\ge 1$, then
$\nu_5(\gcd(h, 60)) = 1$, so $\nu_5(h')=\nu_5(h)- 1$.
For $h$ divisible by $5$, $h+3\equiv3\pmod5$, so
$\nu_5\left(\frac{h(h+3)}{2}\right)=\nu_5(h)$, and taking
the gcd with $\frac{h(h+3)}{2}$ doesn't change $\nu_5(h)$.
Again, for $h$ divisible by $5$,
$18h^4 + 135h^3 + 400h^2 + 585h + 422\equiv 2\pmod5$, while
$\nu_5(120)=1$, so $\nu_5$ of the big
fraction, as well as $\nu_5(h'')$, is $\nu_5(h)-1$,
which agrees with $\nu_5(h')$.

Finally, suppose $p\ge 7$.  Then $2$, $60$, and $120$ are all
relatively prime to $p$.  If $\nu_p(h)=0$, then clearly
$\nu_p(h')=\nu_p(h'')=0$.  If $\nu_p(h)\ge 1$, then
$\nu_p(\gcd(h, 60)) = 0$, so $\nu_p(h')=\nu_p(h)$.
On the other hand, if $\nu_p(h)\ge 1$, then
\[
\begin{aligned}
&\nu_p\left(\frac{h(h+3)}{2}\right) = \nu_p(h),\text{ and }\\
&\nu_p\left(\frac{h(18h^4 + 135h^3 + 400h^2 + 585h + 422)}{120}\right)
  \ge \nu_p(h),
\end{aligned}  
\]
so $\nu_p(h'')=\nu_p(h)=\nu_p(h')$. This concludes the proof.
\end{proof}

We now want to give an elementary but rigorous proof of the
Braun-Douglas Theorem for $G_2$.  We start with analysis of the
odd torsion.  For convenience in what follows, if $x$ is
a positive integer, let $x_\odd$ denote the maximal odd factor
of $x$.  Of course, $x_\odd = \prod_{p\text{ prime }\ge 3} p^{\nu_p(x)}$.
\begin{theorem}
\label{thm:G2oddtorsion}
Let $h$ be a positive  integer, viewed as a twisting class on $G_2$.
Then $K_\bullet(G_2, h)$ is a finite torsion group in all degrees.
Its odd torsion {\lp}in any degree{\rp} is cyclic of order
\[
c(G_2, h)_\odd = h_\odd/\gcd(h_\odd, 15).
\]
\end{theorem}
\begin{proof}
We use the fibration \cite[Lemme 17.1]{MR0064056}
\[
\SU(2) \to G_2 \to V_{7,2},
\]
and get from Theorem \ref{thm:SegalSS} a spectral sequence
\[
E^2_{p,q} = H_p(V_{7,2}, K_q(\SU(2), h)) \Rightarrow K_\bullet(G_2, h).
\]
(The restriction map $H^3(G_2)\to H^3(\SU(2))$ is an isomorphism.)
Here $K_q(\SU(2), h)\cong \bZ/h$ for $q$ even and is $0$ for
$q$ even.  Since $E^2$ is torsion, so is $K_\bullet(G_2, h)$.
The Stiefel manifold $V_{7,2}$ is $11$-dimensional and
has only one nontrivial homology group below the top
dimension, namely a $\bZ/2$ in dimension $5$
(see for example \cite[\S3.D]{MR1867354}), so
after inverting $2$, $V_{7,2}$ becomes homotopy equivalent
to $S^{11}$ by the Hurewicz Theorem modulo 
the Serre class of $2$-primary torsion groups.  Thus
from the point of view of odd torsion, we are in the situation
of Theorem \ref{thm:SegalSSdiff} with a unique differential
$d^{11}$.
We have the long exact homotopy sequence
\[
\pi_{11}(G_2) \to \pi_{11}(V_{7,2}) \to \pi_{10}(\SU(2)) \to \pi_{10}(G_2),
\]
and $\pi_{10}(\SU(2)) \cong \bZ/15$, $\pi_{10}(G_2)=0$
\cite{MR0206958}. Thus the boundary map
$\pi_{11}(V_{7,2}) \to \pi_{10}(\SU(2))$ has kernel of order $15$,
and the theorem follows from Theorem \ref{thm:SegalSSdiff},
exactly as in the proof of Theorem \ref{thm:SU3}.

Let's first deal with the $5$-primary torsion. If $\gcd(h,5)=1$, then
$K_\bullet(G_2, h)$ can't have $5$-primary torsion, by Theorem
\ref{thm:twistedKish-tors}.   So assume $h$ is divisible by $5$ and
localize everything at $5$.  Once again, let's use the notation of
Theorem \ref{thm:Hur}.  The first $5$-primary torsion in the
homotopy groups and homology groups of $P_1$ occurs in degree $10$.
So the Hurewicz map
$\pi_{10}(S^3)\cong \pi_{10}(P_1)\to H_{10}(P_1)\cong \bZ/5$ is a
$5$-local isomorphism, as is the Hurewicz map to $\widetilde{ku}_{10}(P_1)$.
Just as in the proof of Theorem \ref{thm:Hur}, we need to show that
the map $\widetilde{ku}_{10}(P_1)\to \widetilde{ku}_{10}(P_h)$ is
injective on the $5$-torsion.  And again, we do this by comparing
the Segal spectral sequence
\[
H_p(K(\bZ/h, 2), ku_q(P_1))\Rightarrow ku_\bullet(P_h)
\]
with the AHSS for computing $ku_\bullet (K(\bZ/h, 2))$.  (Here
everything is localized at the prime $5$.)  This is exactly like
the $3$-primary calculation in Theorem \ref{thm:Sp2}.

The result for $3$-primary torsion
follows from Theorem \ref{thm:ANSS} below.
\end{proof}

\begin{theorem}
\label{thm:G2twotorsion}
Let $h$ be a positive  integer, viewed as a twisting class on $G_2$.
Then $K_\bullet(G_2, h)$ is a finite torsion group in all degrees.
Its $2$-primary torsion {\lp}in any degree{\rp} is cyclic of order
\[
c(G_2, h)_{2\text{-primary}} = 2^{\max(0,\nu_2(h)-2)}.
\]
In other words,
\[
\nu_2\left(c(G_2, h)\right)
= \begin{cases}
  0, &\nu_2(h)\le 2,\\
  \nu_2(h)-2, &\nu_2(h)> 2.
\end{cases}
\]
\end{theorem}
\begin{proof}
First suppose that $\nu_2(h)\le 1$.  This time we use the fibration
\[
\SU(3)\to G_2\to S^6,
\]
coming from the action of $G$ on the unit sphere of the
imaginary octonians.  We can apply Theorem \ref{thm:SU3}
together with Theorems \ref{thm:SegalSS} and \ref{thm:SegalSSdiff}.
The inclusion $\SU(3)\hookrightarrow G_2$
induces an isomorphism on $H^3$, and $K_\bullet(\SU(3), h)$
has no $2$-torsion, and that proves the theorem in this case.
Note that in the case $\nu_2(h)= 1$, we see that the picture for
the Segal spectral sequence attached to
\[
\SU(2) \to G_2 \to V_{7,2}
\]
has to look like Figure \ref{fig:G2}, with the red arrows
isomorphisms, so that everything cancels out.
\begin{figure}[hbt]
\[
\xymatrix@!0{
&K_\bullet(\SU(2),h)&&&&&&&&&&&&&&\\
4&\bullet\ar[u]&&&&&\cdot&\cdot&&&&&&&&\\
3&&&&&&&&&&&&&&&&\\
2&\bullet\ar@{-}[uu]&&&&&\cdot&\cdot&&&&&&&&\\
1&&&&&&&&&&\\
0&\bullet\ar@{-}[rrrrr] \ar@{-}[uu] &&&&
&\cdot\ar@{-}[r]\ar@[red][llllluuuu]&\cdot\ar@{-}[rrrrr]&&&&&
\bullet\ar[rr]\ar@[red][llllluuuu]
&& H_\bullet(V_{7,2})\\
&0&1&2&3&4&5&6&7&8&9&10&11&&
}\]
\caption{The Segal SS for $2$-primary twisted $K$-homology of $G_2$
  when $\nu_2(h)\ge 1$. Heavy dots indicate copies of $\bZ/2^{\nu_2(h)}$.
  Light dots indicate copies of $\bZ/2$. The diagonal red
  arrow shows the differential $d^5$.}
\label{fig:G2}
\end{figure}
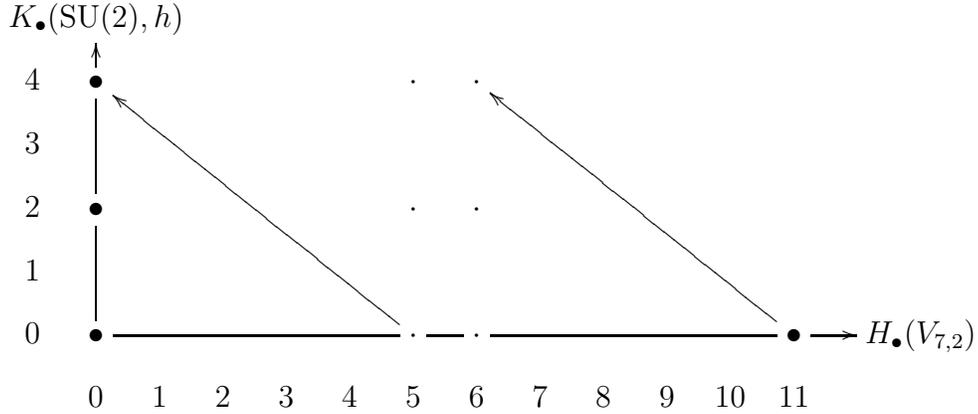

If $\nu_2(h)\ge 2$, then in Figure \ref{fig:G2}, the red $d^5$ arrows
will still be non-zero for the same reasons as before\footnote{The
  map $E^5_{5,0}\to E^5_{0,4}$ is determined by Theorem \ref{thm:SegalSSdiff}
  and the calculation of the Hurewicz map in twisted $K$-homology
  from Theorem \ref{thm:Hur}.  The other map $E^5_{11,0}\to E^5_{6,4}$
  is linked to this one by the module action of $K_\bullet(G_2)$, which is
  an exterior algebra over $\bZ$, on $K_\bullet(G_2, h)$.},
but this
time the copies of $\bZ/2^{\nu_2(h)}$ are reduced to $\bZ/2^{\nu_2(h)-1}$
at the $E^6$ stage.  At this point the dots in the $H_5$ and $H_6$
columns have disappeared and the spectral sequence now looks
like one for a fibration with $S^{11}$ as the base and with
$\bZ/2^{\nu_2(h)-1}$ in even degrees in the twisted $K$-homology of the
fiber.  There is still the differential
\[
d^{11}\co E^{11}_{11, j}\to E^{11}_{0, j+10}
\]
to be reckoned with.  We claim that this differential has
image isomorphic to $\bZ/2$, which will give the desired result.

This comes about in the following way.  If $G_2$ really fit into a
fibration $S^3\to G_2\to S^{11}$ and we were looking at the
associated $2$-local Segal spectral sequence for computing $K_\bullet(G_2, h)$,
then $d^{11}$ would vanish as a consequence of Theorem
\ref{thm:SegalSSdiff}, since $\pi_{10}(S^3)$ has no $2$-torsion.
But in our situation, the groups $\bZ/2^{\nu_2(h)-1}$ in
$E^{11}_{11, 2k}$ and in $E^{11}_{0, 2k+10}$ are really different.
The former arose as \emph{kernels} of $d^5$, a map
$\bZ/2^{\nu_2(h)}\to \bZ/2$, and the latter as \emph{cokernels}
of a map $\bZ/2\to \bZ/2^{\nu_2(h)}$ (see Figure \ref{fig:G2} again).
By \cite{MR0214099}, $K_\bullet(G_2)\cong \bigwedge(x_3,x_{11})$,
and exterior algebra on two odd generators, and the module action of
$x_{11}\in K_\bullet(G_2)$ on $K_\bullet(G_2,h)$ and on the spectral
sequence sets up an isomorphism $E^2_{0,2k}\to E^2_{11,2k}$
which has to pass to an isomorphism
$K_{2k}(G_2,h)\to K_{2k+11}(G_2,h)$. This can only happen if $E^2_{0,2k}$ and
$E^2_{11,2k}$ are each quotients of an index-$2$ subgroup of
$K_{2k}(S^3, 2^{\nu_2(h)})\cong \bZ/2^{\nu_2(h)}$ by a subgroup of
order $2$, and we end up with a $2$-torsion subgroup of order
$2^{\nu_2(h)-2}$.

An alternative way to prove the theorem when $\nu_2(h)\ge 2$ is to
use the Segal spectral sequence for the fibration
$\SU(3)\to G_2\to S^6$ to compute the order of $KU_\bullet(G_2, h)$,
along with the previous argument with the other spectral sequence to
deduce that the group in each degree is cyclic.  For example, 
suppose $h=4$ and consider the differential
$H_6(S^6, K_0(\SU(3), h))\to H_0(S^6, K_5(\SU(3), h))$.
Both groups here are cyclic of order $2$, and the generator of
the domain is the image of the Hurewicz map from $\pi_6(S^6)$.
Since $\pi_5(\SU(3))\cong \bZ$ \cite{MR0169242}
and $\pi_5(G_2)=0$ \cite{MR0206958}, the boundary map
$\partial\co \pi_6(S^6)\to \pi_5(\SU(3))$ is an isomorphism
and we just need to determine what happens to the generator of
$\pi_5(\SU(3))$ under the Hurewicz map to $K_5(\SU(3), h)$.
But the Hurewicz map $\pi_5(\SU(3))\to H_5(\SU(3))$ has image
which is of index $2$ \cite[(4.3)]{MR0108790}, so the image
of $\pi_5(\SU(3))$ in $K_5(\SU(3), h)$ corresponds exactly to the
kernel of
\[
d^5\co H_5(S^5, K_0(\SU(2), 4))\to H_0(S^5, K_4(\SU(2), 4))
\]
in the Segal spectral sequence for $K_\bullet(\SU(3), 4)$
from the fibration $S^3\to \SU(3) \to S^5$, and this
is the entire group $K_5(\SU(3), h)\cong \bZ/2$.
Thus in the Segal spectral sequence, the differential
\[
H_6(S^6, K_0(\SU(3), h))\to H_0(S^6, K_5(\SU(3), h))
\]
is an isomorphism.  The differential with the other parity
can also be seen to be an isomorphism, and in this way one 
can show that $K_\bullet(G_2, 4)=0$.
\end{proof}

\subsection{Analysis of Hurewicz maps}
\label{sec:Hur}
To complete all the theorems from Section \ref{sec:twistedKrank2},
we need the following technical result.
\begin{theorem}
\label{thm:ANSS}
Let $h$ be a positive integer and let $P_h$ be the principal
$\bC\bP^\infty$-bundle over $S^3$ classified by a map of degree $h$
from $S^3$ to $K(\bZ,3)$, as in the proof of Theorem \ref{thm:Hur}.
Then the Hurewicz maps $\pi_j(S^3)\cong \pi_j(P_h)\to ku_j(P_h)$
are injective on the
\begin{enumerate}
\item $2$-torsion $\bZ/2$ when $j=4$ and $h$ is even {\lp}this case was in
 \textup{Theorem \ref{thm:Hur}}{\rp},
\item $2$-primary torsion $\bZ/4$ when $j=6$ and $h$ is divisible by $4$ ---
  if $h\equiv 2\pmod 4$, the map is non-zero {\lp}these cases were
  needed for  \textup{Theorem \ref{thm:Sp2}}{\rp},
\item $3$-torsion $\bZ/3$ when $j=10$ and $3$ divides $h$ {\lp}this case was
  needed for  \textup{Theorem \ref{thm:G2oddtorsion}}{\rp},
\item $5$-torsion $\bZ/5$ when $j=10$ and $5$ divides $h$ {\lp}this case was
  also needed for  \textup{Theorem \ref{thm:G2oddtorsion}}{\rp} . 
\end{enumerate}
In all cases, the image maps injectively under the quotient
map $K_0(P_h)\to K_0(S^3, h)$.
\end{theorem}
Before starting the proof, let's explain a bit about strategy.
We have lumped all of these results together and included cases (1) and (4)
(even though we already did those by another method) to illustrate a
common method of attack using either the Adams-Novikov Spectral
Sequence (ANSS) or the classical Adams Spectral
Sequence (ASS). For this we localize at the appropriate prime
($2$, $3$, or $5$).  We will use the fact that the unstable
Hurewicz homomorphism $\pi_j(P_h)\to ku_j(P_h)$ factors through the
stable Hurewicz homomorphism $\pi_j^s(P_h)\to ku_j(P_h)$, and the latter
can be computed by comparing $\pi^s_\bullet$ and $ku_\bullet$
with the help of the ANSS or the ASS.
\begin{proof}
We localize at the appropriate prime. It
suffices to look at the associated Brown-Peterson homology $BP_\bullet$,
since $MU_\bullet$ splits as a wedge of shifted copies of $BP$ and
$K_\bullet$ can be recovered from $MU_\bullet$ by the Conner-Floyd isomorphism
$MU_\bullet(P_h)\otimes_{MU_\bullet} K_\bullet \cong K_\bullet(P_h)$. The ANSS
has the form (\cite{MR860042,0802.1006} or
\cite[Part III, \S15]{MR1324104})
\[
\Ext^s_{BP_\bullet BP} (\Sigma^t BP_\bullet,\widetilde{BP}_\bullet (P_h))
\Rightarrow \{S^{t-s}, P_h\},
\]
and the edge homomorphism
\[
\{S^t,P_h\}\to \Hom_{BP_\bullet BP}(\Sigma^t BP_\bullet ,\widetilde{BP}_\bullet (P_h))
\]
is the stable Hurewicz map in $BP$-homology.  One can compare this
with study of the classical Adams spectral sequence (ASS)
\[
\Ext^s_{\cA_*} (\Sigma^t \bF_p,\tH_\bullet (P_h;\bF_p))
\Rightarrow \{S^{t-s}, P_h\},
\]
for which the edge homomorphism
\[
\{S^t,P_h\}\to \Hom_{\cA_*}(\Sigma^t \bF_p ,\tH_\bullet (P_h;\bF_p))
\]
is the stable Hurewicz map in ordinary mod-$p$ homology. (We have
deliberately ignored a few localization and completion issues which
don't cause problems in our case. Here $\cA_*$ is the dual
Steenrod algebra at the prime $p$.)

\subsubsection*{Proof of case (4)}
Let's start with case (4), taking $p=5$, starting with ordinary homology.
We have $H^\bullet(P_h;\bF_5)\cong \bF_5[u]\otimes \bigwedge(y)$,
where $u$ is in degree $2$ and $y$ is in degree $5$.
The Bockstein $\beta_{\nu_5(h)}$ sends $u$ to $y$, and $P^j(u^j)=u^{5j}$.
In particular, there is no non-zero $\cA$-module map
$H^\bullet(P_h;\bF_5)\to \Sigma^{10}\bF_5$, since any such map would send
$u$ to $0$ and $P^1u$ to something non-zero, and thus the Hurewicz map
in $\bF_p$-homology (which would be dual to this map in cohomology)
has to be $0$.  Of course, we could also observe this from the fact that
$H_{10}(P_h;\bZ)\cong \bZ/(5h)$, which has $5$-primary subgroup cyclic
of order $5^{\nu_5(h)+1}\ge 25$, and since the $5$-primary subgroup of
$\pi_{10}(P_h)\cong \pi_{10}(S^3)\cong \bZ/15$ is cyclic of order $5$,
the image of the Hurewicz map has to reduce mod $5$ to $0$.
But in fact the \emph{integral} Hurewicz map in degree $10$ vanishes,
for any $\nu_5(h)\ge 1$;
one way to see this is to use the $5$-local Serre spectral sequence of
the fibration $P_1\to P_h\to K(\bZ/h, 2)$ as in Figure \ref{fig:SSPost}.
In the range of dimensions we're interested in,
$H^\bullet(K(\bZ/5^{\nu_5(h)},2);\bF_5)$ agrees with
\[
\bF_5[\iota_2, \beta_{\nu_5(h)}P^1\beta_{\nu_5(h)}\iota_2]
\otimes \bigwedge(\beta_{\nu_5(h)}\iota_2, P^1\beta_{\nu_5(h)}\iota_2).
\]
The generators here have degrees $2,12,3,11$, respectively.
Via the calculation of the Bockstein spectral sequence in
\cite[Theorem 10.4]{MR0281196},
\begin{multline*}
  H^{11}(K(\bZ/5^{\nu_5(h)},2);\bZ)
\cong H_{10}(K(\bZ/5^{\nu_5(h)},2);\bZ)\\ \cong (\bZ/5^{\nu_5(h)})
(\iota_2)^4(\beta_{\nu_5(h)}\iota_2)\oplus (\bZ/5) P^1\beta_{\nu_5(h)}\iota_2.
\end{multline*}
The $k$-invariant of the $5$-local Postnikov approximation
$K(\bZ/5, 10) \to X_1\to K(\bZ/5^{\nu_5(h)},2)$ to $P_h$ can be
identified with the image under $d_{11}$ of the canonical generator of
$H^{10}(K(\bZ/5, 10);\bF_5)$ in the Serre spectral sequence for this
Postnikov approximant, and has to be
non-zero, since otherwise the $5$-primary torsion in $H_{10}(P_h;\bZ)$
would be $\bZ/5\oplus H_{10}(K(\bZ/5^{\nu_5(h)},2))\cong
(\bZ/5)^2 \oplus (\bZ/5^{\nu_5(h)})$,
not $\bZ/5^{\nu_5(h)+1}$, so the $k$-invariant can be seen to
be a non-zero multiple of $P^1 (\beta_{\nu_5(h)}\iota_2)$,
and the Hurewicz map (which corresponds to the image of
$H_{10}(K(\bZ/5, 10))$ under the edge homomorphism) has to vanish.

On the other hand, consider the ANSS. The generators of
$BP_\bullet$ are $v_1$ in degree $2(p-1)=8$, $v_2$ in degree $2(p^2-1)=48$,
etc. Since these are all in even degree and the homology of $P_h$
is also all in even degree, the AHSS for $BP_\bullet$ collapses and
$BP_{\text{odd}}(P_h)$ vanishes identically.  Since we've already seen that
the Hurewicz map $\pi_{10}(P_h)\to H_{10}(P_h)$ vanishes,
the image of the Hurewicz map $\pi_{10}(P_h)\to BP_\bullet(P_h)$
has to map to $0$ in $E^\infty_{10,0}\cong H_{10}(P_h)$ and thus
has to lie in $E^\infty_{2,8}\cong (\bZ/5^{\nu_5(h)})v_1$ (here the
indexing of $E^\infty$ corresponds to the AHSS for $BP_\bullet$).
Note that $BP_\bullet(P_h)$ does not necessarily split as a direct
sum of $BP_\bullet BP$-comodules corresponding to the summands
of $E^\infty$ for the AHSS, but it has a filtration for which this is
the associated graded $BP_\bullet BP$-comodule.

Note that $BP_\bullet/(hj)=BP_\bullet/p^{\nu_p(h)+\nu_p(j)}$, so
we get a spectral sequence converging to the $E^2$-term of the
ANSS for which $E^1$ is a sum of copies of
\[
\Ext^{s,t-2j}_{BP_\bullet BP}(BP_\bullet,BP_\bullet/p^{\nu_p(h)+\nu_p(j)}),
\quad 2j\le t.
\]
Under the map $BP_\bullet(P_h)\to K_0(S^3, h)$, $1$ and $v_1$ map to $1$
and the other generators $v_j$, $j\ge 2$, map to $0$.  So we are
particularly interested in what happens in low topological degree
($t-s=10$ for case (4) of the theorem, other values no larger
than this for the other cases) and with regard to $v_1$.

For our case at hand with $p=5$, where $\nu_5(h)=1$ for simplicity,
a diagram of the $\Ext$ groups may be
found in \cite[Figure 4.4.16]{MR860042}.  In low degrees
\cite[Theorem 4.4.15]{MR860042},
$\Ext^{\bullet,\bullet}_{BP_\bullet BP}(BP_\bullet,BP_\bullet/p)$ is a polynomial
algebra on $v_1$ (which has bidegree $s=0$, $t=8$)
tensored with an exterior algebra on $h_{1,0}$
(which has bidegree $s=1$, $t=8$).  We see that not very much can
contribute, except for
\[
\Ext^{0}_{BP_\bullet BP}(\Sigma^{10}BP_\bullet, \Sigma^2 BP_\bullet/5)
\cong \bF_5v_1
\]
corresponding to
$E^\infty_{2,8}$ in the AHSS (which is where we expected the Hurewicz
homomorphism to land).  This can't be killed by a differential,
so the Hurewicz map is non-zero, and since $v_1\mapsto 1$, this
maps to an element of order $5$ in $K_0(S^3, h)$.

\subsubsection*{Proof of case (3)}
The other cases of the theorem are treated in a similar fashion.
Let's next deal with the other odd torsion case, (3), with $p=3$
and again degree $10$. We'll take $\nu_3(h)=1$ (again for
simplicity---when $\nu_3(h)$ is larger, things are similar but the
bookkeeping is more complicated).  Again, the Hurewicz map
$\pi_{10}(P_h)\to H_{10}(P_h;\bF_3)$ vanishes since if there were
a map $f\co S^{10}\to P_h$ which were non-zero on homology with
$\bF_3$ coefficients, the dual map on cohomology would send the 
generator $u\in H^2(P_h;\bF_3)$ to $0$, and thus would have
to kill $u^5$, which is
the generator in degree $10$.  So once again we look at the ANSS
to study the Hurewicz map in $BP$ homology.
This time, $v_1$ is in degree $2\cdot(3-1) = 4$,
$v_2$ in degree $2\cdot(3^2-1)=16$, etc., so the Hurewicz map
in $BP$ homology will have target in a ${BP_\bullet BP}$-subcomodule
$M$ of $BP_\bullet(P_h)$ which is an extension
\[
0\to \Sigma^2 BP_\bullet/3 \to M \to \Sigma^6 BP_\bullet/9\to 0,
\]
where the subobject comes from $H_2(P_h)\cong \bZ/3$ and the
quotient comes from $H_6(P_h)\cong \bZ/9$.  This extension
is nontrivial since in cohomology, $P^1$ is nonzero from
$H^2(P_h;\bF_3)$ to $H^6(P_h;\bF_3)$.  We get a long exact sequence
of $\Ext$ groups (all over ${BP_\bullet BP}$, which we omit for
conciseness):
\begin{multline*}
0\to \Ext^{0,10}(BP_\bullet, \Sigma^2 BP_\bullet/3) \to
\Ext^{0,10}(BP_\bullet, M) \to \\
\Ext^{0,10}(BP_\bullet, \Sigma^6 BP_\bullet/9)
\to \Ext^{1,10}(BP_\bullet, \Sigma^2 BP_\bullet/3) \to \cdots.
\end{multline*}
Here $v_1^2$ gives a non-vanishing contribution to
$\Ext^{0,10}(BP_\bullet, M)$ which can't be killed
under any differential of the ANSS.  The upshot of all of this
is that the Hurewicz map in $BP$-homology is non-zero
$\pi_{10}(P_h)\to BP_{10}(P_h)$, and that under the map to $K_0(S^3,h)$,
this goes to non-zero $3$-torsion.

\subsubsection*{Proof of case (1)}
Now let's consider cases (1) and (2), which involve the prime $p=2$.
First consider case (1), which is relatively easy; we want to compute
the Hurewicz map in $BP$ in degree $4$ for $P_h$, $h$ even, using the
ANSS.  This time the generators are $v_1$ in degree $2$,
$v_2$ in degree $6$, etc., and
$\Ext^{0}_{BP_\bullet BP}(\Sigma^4BP_\bullet, \widetilde{BP}_\bullet(P_h))$
potentially has contributions from
\[
\Ext^{0,2}(BP_\bullet,BP_\bullet/2^{\nu_2(h)})\quad\text{and}\quad
\Ext^{0,0}(BP_\bullet,BP_\bullet/2^{\nu_2(h)+1}).
\]
Since the Hurewicz map vanishes in ordinary homology, the composite
$\pi_4(P_h)\to BP_4(P_h)\to H_4(P_h)$ (the last map being the edge
homomorphism of the AHSS) has to vanish, so we are only interested
in the first term.  Say that $\nu_2(h)=1$; then the
picture of $\Ext^{s,t}(BP_\bullet,BP_\bullet/2)$ is shown in
\cite[Figure 4.4.32]{MR860042}.  Our candidate for the image of the
Hurewicz map is $v_1\in \Ext^{0,2}$; this is a permanent cycle as one
can see from the picture, so the Hurewicz map is non-zero.  And
$v_1$ reduces to $1$ in $K_0(S^3, h)$.  

\subsubsection*{Proof of case (2)}
Finally we have the case (2) in topological degree $6$. First take
$\nu_2(h)=1$; then $H^\bullet(P_h;\bF_2)=\bF_2[u]\otimes\bigwedge \Sq^1 u$,
with the polynomial generator $u$ in degree $2$.  The ordinary Hurewicz
map has to vanish, since there is no non-zero ring homomorphism
$H^\bullet(P_h;\bF_2)\to H^\bullet(S^6, \bF_2)$.  So candidates for
the $BP$ Hurewicz map have to live in
in a ${BP_\bullet BP}$-subcomodule
$M$ of $BP_\bullet(P_h)$ which is an extension
\[
0\to \Sigma^2 BP_\bullet/2 \to M \to \Sigma^4 BP_\bullet/4\to 0,
\]
where the subobject comes from $H_2(P_h)\cong \bZ/2$ and the
quotient comes from $H_4(P_h)\cong \bZ/4$.
Once again the contribution of $v_1^2\in \Ext^{0,4}$ to
$\Ext^{0,6}(BP_\bullet, BP_\bullet(M))$ is a permanent cycle mapping nontrivially
to $K_0(S^3, h)$.

\subsubsection*{An alternate method}
Before we deal with higher $p$-primary torsion, we should mention another
approach to our theorem
using the classical ASS, which is discussed in this context in
\cite[Part III, \S16]{MR1324104}. To avoid unnecessary repetitions,
we go into detail only with $p=2$ and cases (1) and (2) of the theorem.
Following Adams' notation, let $\cB$ be the subalgebra of the mod-$2$
Steenrod algebra $\cA$ generated by $\Sq^1$ and
$Q_1 = \Sq^1\Sq^2 + \Sq^2\Sq^1$.  This is an exterior algebra on
generators of degrees $1$ and $3$, so it has total dimension $4$.
By a change-of-rings argument, Adams
\cite[Part III, Proposition 16.1]{MR1324104} proves that
the ASS for $\widetilde{ku}_\bullet(X)$ has $E_2$ term which
simplifies to $\Ext^{s,t}_{\cB_*}(\bF_2,\tH_\bullet(X;\bF_2))$.
We can study the Hurewicz map $\pi^s_\bullet(X)\to ku_\bullet(X)$
by comparing this ASS with the one with $E_2$ terms
$\Ext^{s,t}_{\cA_*}(\bF_2,\tH_\bullet(X;\bF_2))$ converging to
$\pi^s_\bullet(X)$.  The natural map $\Ext^{s,t}_{\cA_*}\to \Ext^{s,t}_{\cB_*}$
comes from the forgetful functor from ${\cA_*}$-comodules to
${\cB_*}$-comodules.
The advantage of this approach, applied to $X=$ the suspension spectrum
of $P_h$, is
that we know $H^\bullet(P_h;\bF_2)$ quite explicitly as a module over $\cA$
(and in particular over $\cB$).  Indeed, if $h$ is even, in the
Serre spectral sequence for computing $H^\bullet(P_h;\bF_2)$ from
$\bC\bP^\infty\to P_h\to S^3$, the only differential $d_3$ vanishes,
and so $H^\bullet(P_h;\bF_2)=\bF_2[u]\otimes \bigwedge(y)$, where $u$
is in degree $2$ and $y$ is in degree $3$. Since $H^2(P_h;\bZ)=0$ and
$H^3(P_h;\bZ)\cong\bZ/h$, if $\nu_2(h)=1$, $\Sq^1u=y$, whereas
if $\nu_2(h)>1$, $\Sq^1u=0$ and $\beta_{\nu_2(h)}(u)=y$. In both cases
we have $\Sq^2u=u^2$, $\Sq^jy=0$ for $j\ge 1$. (The last 
identity follows from the fact that $y$ is pulled back from
$H^3(S^3;\bF_2)$, on which $\cA$ acts trivially.)  The rest of the
action of the Steenrod algebra can be determined from the Cartan
relations.  For the sake of definiteness, let's take $\nu_2(h)=1$.
Note that the inclusion $\bC\bP^\infty\hookrightarrow P_h$
induces an isomorphism of $\bF_2[u]\subset H^\bullet(P_h;\bF_2)$ onto
$H^\bullet(\bC\bP^\infty;\bF_2)$.  So if $\cAe$ is the subalgebra of
$\cA$ generated by the $\Sq^{2^j}, \, j\ge 1$, and $a\in \cAe$,
then $a(u)$ must be a linear combination of primitive elements
of $\bF_2[u]$, i.e., of the elements $u^{2^j}, \, j\ge 0$, by the same
argument found in \cite[pp.\ 19--21]{MR0196742}.

Most of the work in computing the $\Ext$ and stable homotopy groups
was done by Liulevicius \cite{MR0156346} and Mosher \cite{MR0227985}.
Let $M$ be the left $\cA$-module $\tH^\bullet(\bC\bP^\infty; \bF_2)$,
and let $N$ be the left $\cA$-module $\tH^\bullet(Y; \bF_2)$,
where $Y$ is the result of attaching a $3$-cell to $\bC\bP^\infty$
via a map $S^2\xrightarrow{h} \bC\bP^1\subset \bC\bP^\infty$ of degree $h$.
$Y$ can be indentified with a subcomplex of $P_h$ and the cofiber
of the inclusion $Y\to P_h$ can be identified with $\Sigma^3\bC\bP^\infty$.
So we have exact sequences of $\cA$-modules
\begin{equation}
  \begin{aligned}
  \text{(a)}\qquad 0\to \Sigma^3\bF_2 &\to N \to M\to 0,\\  
  \text{(b)}\qquad0\to \Sigma^3M &\to \tH^\bullet(P_h;\bF_2) \to N\to 0.
  \label{eq:cohomPh}
  \end{aligned}
\end{equation}
These extensions are nontrivial since we have the relations
$\Sq^1(u^{2j+1})=u^{2j}y$, and so there are classes 
$v\in \Ext^{1,3}_{\cA}(M,\bF_2)$, $w\in \Ext^{1,3}_{\cA}(N,M)$, 
associated to \eqref{eq:cohomPh} (a) and (b), respectively.

In low dimensions $\Ext^{s,t}_{\cA}(M,\bF_2)$ was computed in
\cite{MR0156346}, and there is only one Adams differential in this
range. There is a unique nonzero element in $\Ext^{1,3}_{\cA}(M,\bF_2)$,
so that is $v$, and the connecting map in the long exact sequence
coming from \eqref{eq:cohomPh}(a) is Yoneda product with $v$
by \cite[Theorem 2.3.4]{MR860042}.  From knowledge of
$\Ext_\cA^{s,t}(\bF_2,\bF_2)$ \cite[Theorem 3.2.11]{MR860042}
and of $\Ext^{s,t}_{\cA}(M,\bF_2)$ \cite[Proposition II.3]{MR0156346}
in low dimensions along with the long exact sequence,
we get the diagram of the long exact sequence for
$\Ext^{s,t}_{\cA}(N,\bF_2)$ shown in Figure \ref{fig:ASS0}.  Here
$\Ext_\cA^{s,t}(\bF_2,\bF_2)$ is depicted at the left,
$\Ext^{s,t}_{\cA}(M,\bF_2)$ at the right, and red dots indicate
elements paired under the connecting map (i.e., under product with $v$).

\begin{figure}[hbt]
\begin{tikzpicture}
\draw[->, thick] (-4.5,0) -- (-.5,0);
\draw[->, thick] (-4.5,0) -- (-4.5,3.5);
\draw[thin] (-4.5,0) -- (-3,3);
\draw[thin] (-3,1) -- (-3,3);
\draw[->] (3,0) -- (3,3.5);
\draw[->] (4,1) -- (4, 3.5);
\draw[->] (5,0) -- (5, 3.5);
\node at (0.2,0) {$t-s$};
\node at (6.7,0) {$t-s$};
\node at (-4.5,4.0) {$s$};
\node at (2.0,4.0) {$s$};
\node at (-4.5, -.3) {0};
\node at (-4, -.3) {1};
\node at (-3.5, -.3) {2};
\node at (-3, -.3) {3};
\node at (-2.5, -.3) {4};
\node at (-2, -.3) {5};
\node at (-1.5, -.3) {6};
\node at (-1, -.3) {7};
\node at (2, -.3) {0};
\node at (2.5, -.3) {1};
\node at (3, -.3) {2};
\node at (3.5, -.3) {3};
\node at (4, -.3) {4};
\node at (4.5, -.3) {5};
\node at (5, -.3) {6};
\node at (5.5, -.3) {7};
\node at (-5, 0) {0};
\node at (-5, 1) {1};
\node at (-5, 2) {2};
\node at (-5, 3) {3};
\node at (1.5, 0) {0};
\node at (1.5, 1) {1};
\node at (1.5, 2) {2};
\node at (1.5, 3) {3};
\draw[->, thick] (2,0) -- (6,0);
\draw[->, thick] (2,0) -- (2,3.5);
\draw[fill, red] (-4.5,0) circle [radius=0.05];
\draw[fill, red] (-4.5,1) circle [radius=0.05];
\draw[fill, red] (-4.5,2) circle [radius=0.05];
\draw[fill, red] (-4.5,3) circle [radius=0.05];
\draw[fill] (-4,1) circle [radius=0.05];
\draw[fill] (-3.5,2) circle [radius=0.05];
\draw[fill, red] (-3,1) circle [radius=0.05];
\draw[fill] (-3,2) circle [radius=0.05];
\draw[fill] (-3,3) circle [radius=0.05];
\draw[fill] (-1.5,2) circle [radius=0.05];
\draw[fill] (3,0) circle [radius=0.05];
\draw[fill, red] (3,1) circle [radius=0.05];
\draw[fill, red] (3,2) circle [radius=0.05];
\draw[fill, red] (3,3) circle [radius=0.05];
\draw[fill] (4,1) circle [radius=0.05];
\draw[fill] (4,2) circle [radius=0.05];
\draw[fill] (4,3) circle [radius=0.05];
\draw[thin] (3,0) -- (4.5,1);
\draw[thin] (4.5,1) -- (4.5,2);
\draw[thin] (5,0) -- (5.5,1);
\draw[fill] (4.5,1) circle [radius=0.05];
\draw[fill, red] (4.5,2) circle [radius=0.05];
\draw[fill] (5,0) circle [radius=0.05];
\draw[fill] (5,0) circle [radius=0.05];
\draw[fill] (5,1) circle [radius=0.05];
\draw[fill] (5,2) circle [radius=0.05];
\draw[fill] (5,3) circle [radius=0.05];
\draw[fill] (5.5,1) circle [radius=0.05];
\draw[->>, thick, green] (-4.45,.1) to [out=20,in=160](2.9,1.1);
\draw[->>, thick, green] (-4.45,1.1) to [out=20,in=160](2.9,2.1);
\draw[->>, thick, green] (-4.45,2.1) to [out=20,in=160](2.9,3.1);
\draw[->>, thick, green] (-2.95,1.1) to [out=20,in=160](4.4,2.1);
\end{tikzpicture}  
\caption{The groups $\Ext^{s,t}_{\cA}(\bF_2,\bF_2)$ (left)
  and $\Ext^{s,t}_{\cA}(M,\bF_2)$ (right, following Liulevicius)
  in low dimensions.  Red dots indicate elements
  which cancel under the connecting map
  (green arrows) in the long exact sequence
  for $\Ext^{s,t}_{\cA}(N,\bF_2)$ (for $\nu_2(h)=1$).}
\label{fig:ASS0}
\end{figure}
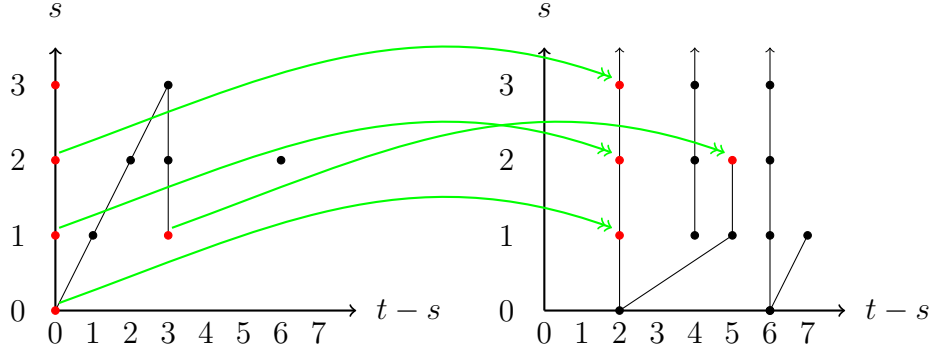

From this picture we can read off the stable homotopy groups of $Y$,
since there are no Adams differentials in the dimension range we're
interested in.  So for example $\pi_4^s(Y)\cong \bZ\oplus \bZ/2$,
with the $\bZ$ coming from $\bC\bP^\infty$ (the $t-s=4$ column
on the right in Figure \ref{fig:ASS0}) and the $\bZ/2$ coming
from $S^3$ (the dot at $t-s=1$, $s=1$ 
on the left in Figure \ref{fig:ASS0}---remember that we shift up in
dimension by $3$). Similarly,  $\pi_6^s(Y)\cong \bZ\oplus \bZ/4$,
with the $\bZ$ coming from $\bC\bP^\infty$ (the $t-s=6$ column
on the right in Figure \ref{fig:ASS0}) and the $\bZ/4$ coming
from $S^3$ (the dots at $t-s=3$, $s=2$ and $3$ 
on the left in Figure \ref{fig:ASS0}). 

To compute
$\Ext_\cA^{s,t}(\tH^\bullet(P_h;\bF_2), \bF_2)$, we need one more
exact sequence coming from \eqref{eq:cohomPh} (b). Since all the
reduced homology of $P_h$ is torsion, so is $\pi^s_\bullet(P_h)$,
and the connecting map $\Ext_\cA^{s,t}(M, \bF_2)\to
\Ext_\cA^{s+1,t+3}(N, \bF_2)$ kills off the $\bZ$ summands.
The Hurewicz maps we are interested in come from the
composites $\pi_j(S^3)\to \pi_j^s(Y) \to \pi_j^s(P_h)$
with $j=4$ and $j=6$, so they come from the $\bZ/2$ in $\pi_4^s(Y)$
and the $\bZ/4$ in $\pi_6^s(Y)$ coming from the dots
on the left in Figure \ref{fig:ASS0} in bidegrees $(s=1, t=2)$, 
resp., $(s=2,t=5)$ and $(s=3, t=6)$. In both 
the cases $j=4$ and $6$, the torsion summand in $\pi_j^s(Y)$ 
cannot be killed by $\pi_{j+1}^s(\Sigma^3\bC\bP^\infty)$.

\begin{figure}[hbt]
\begin{tikzpicture}
\draw[->, thick] (-4.5,0) -- (-.5,0);
\draw[->, thick] (-4.5,0) -- (-4.5,3.5);
\draw[thin] (-2.5,1) -- (-1.5,3);
\draw[thin] (-1.5,2) -- (-1.5,3);
\node at (0.2,0) {$t-s$};
\node at (6.7,0) {$t-s$};
\node at (-4.5,4.0) {$s$};
\node at (2.0,4.0) {$s$};
\node at (-4.5, -.3) {0};
\node at (-4, -.3) {1};
\node at (-3.5, -.3) {2};
\node at (-3, -.3) {3};
\node at (-2.5, -.3) {4};
\node at (-2, -.3) {5};
\node at (-1.5, -.3) {6};
\node at (-1, -.3) {7};
\node at (2, -.3) {0};
\node at (2.5, -.3) {1};
\node at (3, -.3) {2};
\node at (3.5, -.3) {3};
\node at (4, -.3) {4};
\node at (4.5, -.3) {5};
\node at (5, -.3) {6};
\node at (5.5, -.3) {7};
\node at (-5, 0) {0};
\node at (-5, 1) {1};
\node at (-5, 2) {2};
\node at (-5, 3) {3};
\node at (1.5, 0) {0};
\node at (1.5, 1) {1};
\node at (1.5, 2) {2};
\node at (1.5, 3) {3};
\draw[->, thick] (2,0) -- (6,0);
\draw[->, thick] (2,0) -- (2,3.5);
\draw[fill] (-3.5,0) circle [radius=0.05];
\draw[fill] (-2.5,1) circle [radius=0.05];
\draw[fill] (-2,2) circle [radius=0.05];
\draw[fill] (-1.5,2) circle [radius=0.05];
\draw[fill] (-1.5,3) circle [radius=0.05];
\draw[fill, red] (3,0) circle [radius=0.05];
\draw[fill, red] (4,1) circle [radius=0.05];
\draw[fill] (4,0) circle [radius=0.05];
\draw[fill, red] (5,2) circle [radius=0.05];
\draw[fill] (4.5,0) circle [radius=0.05];
\draw[fill] (5,0) circle [radius=0.05];
\draw[->>, thick, green] (-3.45,.1) to [out=20,in=160](2.9,.1);
\draw[->>, thick, green] (-2.45,1.1) to [out=20,in=160](3.9,1.1);
\draw[->>, thick, green] (-1.45,2.1) to [out=20,in=160](4.9,2.1);
\end{tikzpicture}  
\caption{Comparing the $2$-local Adams spectral sequences
  for computing $\pi^s_\bullet(P_h)$ and $ku_\bullet(P_h)$ for
  $h=2k$, $k$ odd. Red dots indicate the contribution
  from $\Ext^{s,2+t}_{\cC_*}(\bF_2,\bF_2)$.  Some other contributions
  on the right are omitted.}
\label{fig:ASS}
\end{figure}
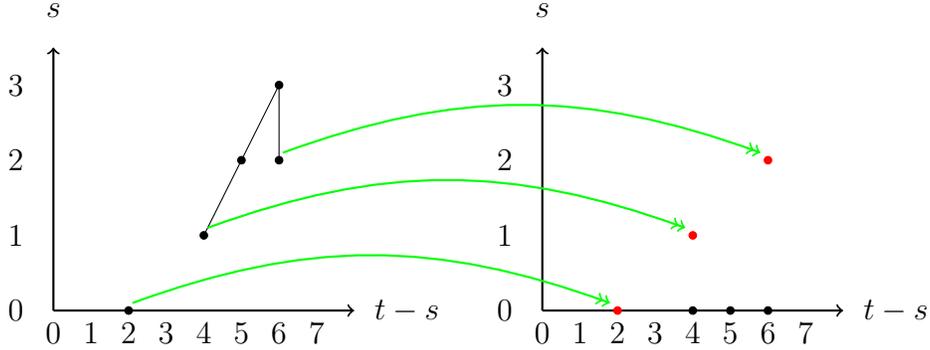

Next let's compute the $\cB$-module structure on $H^\bullet(X)$,
needed for the right side of Figure \ref{fig:ASS}.
When $\nu_2(h)>1$, $\Sq^1$ vanishes
identically on $H^\bullet(P_h;\bF_2)$, and so the $\cB$-module
structure is trivial and
\[
\Ext^{s,t}_{\cB_*}(\bF_2,\widetilde H_\bullet(X;\bF_2))\cong
\widetilde H_t(P_h;\bF_2)\otimes \Ext^{s,*}_{\cB_*}(\bF_2,\bF_2).
\]

When $\nu_2(h)=1$, then $\Sq^1(u^j)=ju^{j-1}y$ and
$\Sq^1(u^jy)=0$, while by an induction using the Cartan formula,
we have $\Sq^2(u^j)=u^{j+1}$ for $j$ odd, $0$ for $j$ even, and
$\Sq^2(u^jy)=u^{j+1}y$ for $j$ odd, $0$ for $j$ even.
Thus
\[
Q_1(u^j)= (\Sq^1\Sq^2 + \Sq^2\Sq^1)(u^j) =
\Sq^1(ju^{j+1}) + \Sq^2(ju^{j-1}y) = 0
\]
in all cases, and similarly $Q_1(u^jy)=0$ in all cases.
Let $\cC = \bigwedge(Q_1)$ be the subalgebra of $\cB$
generated by $Q_1$.  Then we've seen that for $j$ odd,
$u^j$ and $u^{j-1}y$ span a $\cB$-module $M_j$ on which $\Sq^1$ acts
cyclically and $\cC$ acts trivially.  So this module is
$\cB\otimes_{\cC}\bF_2$ and again by change of rings,
$\Ext^{s,2j+t}_{\cB_*}(\bF_2,M_j)\cong
\Ext^{s,2j+t}_{\cC_*}(\bF_2,\bF_2)$.
However, when $j$ is even, $u^j$ and $u^{j-1}y$ each span a
trivial one-dimensional $\cB$-module.  Thus, for $\nu_2(h)=1$,
\begin{multline*}
\Ext^{s,t}_{\cB_*}(\bF_2,\widetilde H_\bullet(P_h;\bF_2))\cong
\bigoplus_{j\text{ odd}} \Ext^{s,2j+t}_{\cC_*}(\bF_2,\bF_2)\\
\oplus \bigoplus_{j\text{ even}} \Ext^{s,2j+t}_{\cB_*}(\bF_2,\bF_2)
\oplus \bigoplus_{j\text{ even}} \Ext^{s,2j+1+t}_{\cB_*}(\bF_2,\bF_2).
\end{multline*}

Note that a simple calculation gives
$\Ext^{s,t}_{\cC_*}(\bF_2,\bF_2)\cong \bF_2$ for all $s\ge0$
and $t=3s$ ($0$ for other values of $t$) and
$\Ext^{s,t}_{\cB_*}(\bF_2,\bF_2)$ is a sum of copies of $\bF_2$,
one for each $s_1,s_2\ge0$ and $s=s_1+s_2$,
$t=3s_1+s_2$ (the formulas for
$t$ come from the fact that $\Sq^1$ raises topological degree by $1$
and $Q_1$ raises topological degree by $3$).  Increasing
$s_2$ corresponds to multiplying by $h_0\in \Ext^{1,1}$.
(This is also all in \cite[Theorem 3.1.16]{MR860042}.)
Thus in case (1) with $\nu_2(h)=1$,
we get in the $E_2$ of the ASS for $\widetilde ku_\bullet(P_h)$
copies of $\bF_2$ in bidegrees
\[
(s,t) = (s, 2+3s), (s_1+s_2,4+3s_1+s_2), (s_1+s_2,5+3s_1+s_2),
\text{ etc}.
\]
These are shown on the right side of Figure \ref{fig:ASS}.
Note that the terms coming from homology in degrees $2j$ and $2j+1$
correspond to the image of $\bZ\beta_j\subset K_0(P_h)$ (in Khorami's
notation in \cite{MR2832567}).  Since $\beta_j$ maps to $0$ in
$K_0(S^3, h)$ for $j\ge 2$, we are really only interested in
the terms with $j=1$, which are indicated by red dots in
Figure \ref{fig:ASS}.  The nontriviality of the green arrows
in Figure \ref{fig:ASS} (which is easy to check purely
algebraically) immediately gives another proof of cases (1) and (2)
when $\nu_2(h)=1$.

\subsubsection*{Proof of cases with $\nu_2(h)>1$}
Finally, we consider cases (1) and (2) when $\nu_2(h)>1$, say
for definiteness $\nu_2(h)=2$.  Then the $\cA$-module extensions
in equation \eqref{eq:cohomPh} now split, and Figure \ref{fig:ASS} is modified
as follows.  On the left-hand side, since $\pi_2(X)=\bZ/4$ (after
localizing at $2$), the columns with $t-s=2,\,3$ are modified
as in \cite[Example 2.1.19]{MR860042}, with the addition of a differential.
This will not matter for us since we only care about topological
degrees $4$ and up.  The other change is that $H_2(P_h)\cong \bZ/4$,
and the $\cB$-submodule of $H^\bullet(X; \bF_2)$ generated by
$u$ and $y$ is now trivial.  That changes the picture on the right as
shown in Figure \ref{fig:ASS1}.  The differentials on the right
are determined by the facts that $H_2(P_h)=\bZ/h$ and that the AHSS
for $ku$ collapses at $E^2$.  This picture proves the remaining cases
of the theorem with $p=2$.  Cases (3) and (4), with $p=3$ or $5$, can
also be handled by the same methods as cases (1) and (2).
The picture analogous to Figure \ref{fig:ASS} for $p=3$, $\nu_3(h)=1$,
and case (3) appears as
Figure \ref{fig:ASS2}. (At an odd prime $p$, $\cC$ becomes the
exterior algebra on $Q_1$, which is of degree $2p-1$.)
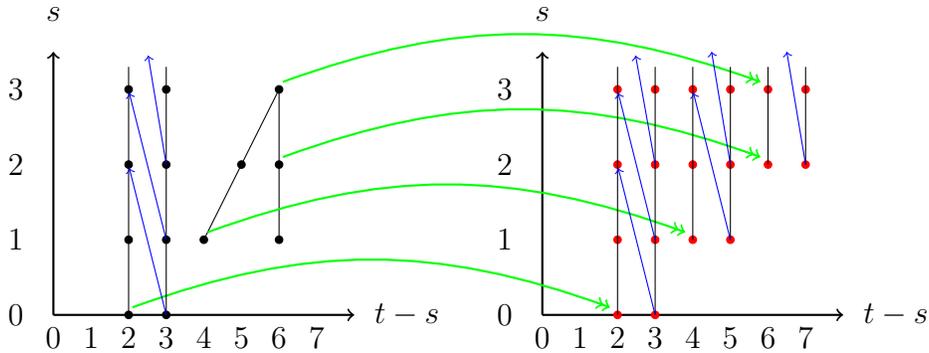
\begin{figure}[hbt]
\begin{tikzpicture}
\draw[->, thick] (-4.5,0) -- (-.5,0);
\draw[->, thick] (-4.5,0) -- (-4.5,3.5);
\draw[thin] (-2.5,1) -- (-1.5,3);
\draw[thin] (-1.5,1) -- (-1.5,3);
\node at (0.2,0) {$t-s$};
\node at (6.7,0) {$t-s$};
\node at (-4.5,4.0) {$s$};
\node at (2.0,4.0) {$s$};
\node at (-4.5, -.3) {0};
\node at (-4, -.3) {1};
\node at (-3.5, -.3) {2};
\node at (-3, -.3) {3};
\node at (-2.5, -.3) {4};
\node at (-2, -.3) {5};
\node at (-1.5, -.3) {6};
\node at (-1, -.3) {7};
\node at (2, -.3) {0};
\node at (2.5, -.3) {1};
\node at (3, -.3) {2};
\node at (3.5, -.3) {3};
\node at (4, -.3) {4};
\node at (4.5, -.3) {5};
\node at (5, -.3) {6};
\node at (5.5, -.3) {7};
\node at (-5, 0) {0};
\node at (-5, 1) {1};
\node at (-5, 2) {2};
\node at (-5, 3) {3};
\node at (1.5, 0) {0};
\node at (1.5, 1) {1};
\node at (1.5, 2) {2};
\node at (1.5, 3) {3};
\draw[->, thick] (2,0) -- (6,0);
\draw[->, thick] (2,0) -- (2,3.5);
\draw[fill] (-3.5,0) circle [radius=0.05];
\draw[fill] (-3.5,1) circle [radius=0.05];
\draw[fill] (-3.5,2) circle [radius=0.05];
\draw[fill] (-3.5,3) circle [radius=0.05];
\draw[fill] (-3,0) circle [radius=0.05];
\draw[fill] (-3,1) circle [radius=0.05];
\draw[fill] (-3,2) circle [radius=0.05];
\draw[fill] (-3,3) circle [radius=0.05];
\draw[fill] (-2.5,1) circle [radius=0.05];
\draw[fill] (-2,2) circle [radius=0.05];
\draw[fill] (-1.5,1) circle [radius=0.05];
\draw[fill] (-1.5,2) circle [radius=0.05];
\draw[fill] (-1.5,3) circle [radius=0.05];
\draw[fill, red] (3,0) circle [radius=0.05];
\draw[fill, red] (4,1) circle [radius=0.05];
\draw[fill, red] (5,2) circle [radius=0.05];
\draw[fill, red] (3.5,0) circle [radius=0.05];
\draw[fill, red] (4.5,1) circle [radius=0.05];
\draw[fill, red] (5.5,2) circle [radius=0.05];
\draw[fill, red] (3,1) circle [radius=0.05];
\draw[fill, red] (3,2) circle [radius=0.05];
\draw[fill, red] (3,3) circle [radius=0.05];
\draw[fill, red] (4,2) circle [radius=0.05];
\draw[fill, red] (4,3) circle [radius=0.05];
\draw[fill, red] (5,3) circle [radius=0.05];
\draw[fill, red] (3.5,1) circle [radius=0.05];
\draw[fill, red] (3.5,2) circle [radius=0.05];
\draw[fill, red] (3.5,3) circle [radius=0.05];
\draw[fill, red] (4.5,2) circle [radius=0.05];
\draw[fill, red] (4.5,3) circle [radius=0.05];
\draw[fill, red] (5.5,3) circle [radius=0.05];
\draw[->>, thick, green] (-3.45,.1) to [out=20,in=160](2.9,.1);
\draw[->>, thick, green] (-2.45,1.1) to [out=20,in=160](3.9,1.1);
\draw[->>, thick, green] (-1.45,2.1) to [out=20,in=160](4.9,2.1);
\draw[->>, thick, green] (-1.45,3.1) to [out=20,in=160](4.9,3.1);
\draw[thin] (-3.5,0) -- (-3.5,3.3);
\draw[thin] (-3,0) -- (-3,3.3);
\draw[thin] (3,0) -- (3,3.3);
\draw[thin] (4,1) -- (4,3.3);
\draw[thin] (3.5,0) -- (3.5,3.3);
\draw[thin] (4.5,1) -- (4.5,3.3);
\draw[thin] (5.5,2) -- (5.5,3.3);
\draw[thin] (5,2) -- (5,3.3);
\draw[->, blue] (-3,0) -- (-3.49, 1.95);
\draw[->, blue] (-3,1) -- (-3.49, 2.95);
\draw[->, blue] (-3,2) -- (-3.24, 3.45);
\draw[->, blue] (3.5,0) -- (3.01, 1.95);
\draw[->, blue] (3.5,1) -- (3.01, 2.95);
\draw[->, blue] (3.5,2) -- (3.24, 3.45);
\draw[->, blue] (5.5,2) -- (5.25,3.5);
\draw[->, blue] (4.5,2) -- (4.25,3.5);
\draw[->, blue] (4.5,1) -- (4.01,2.95);
\end{tikzpicture}  
\caption{Comparing the $2$-local Adams spectral sequences
  for computing $\pi^s_\bullet(P_h)$ and $ku_\bullet(P_h)$ for
  $h=4k$, $k$ odd. Red dots indicate the contribution
  from $\Ext^{s,2+t}_{\cB_*}(\bF_2,\bF_2)\oplus
  \Ext^{s,3+t}_{\cB_*}(\bF_2,\bF_2)$.  Other contributions on the
  right are omitted. Blue arrows show $d_2$.} 
\label{fig:ASS1}
\end{figure}
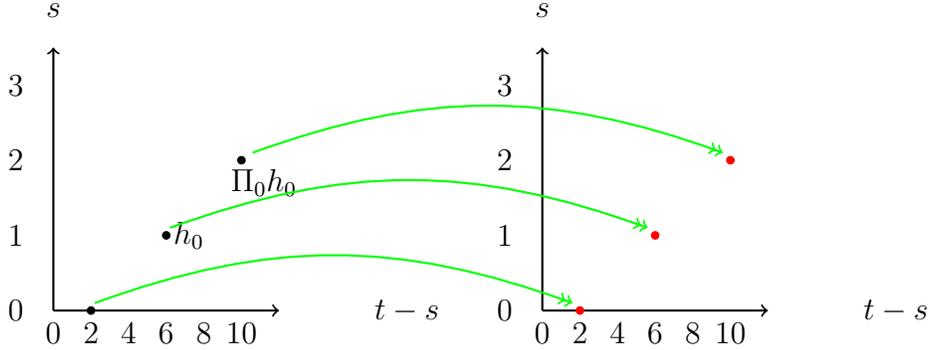
\begin{figure}[hbt]
\begin{tikzpicture}
\draw[->, thick] (-4.5,0) -- (-1.5,0);
\draw[->, thick] (-4.5,0) -- (-4.5,3.5);
\node at (0.2,0) {$t-s$};
\node at (6.7,0) {$t-s$};
\node at (-4.5,4.0) {$s$};
\node at (2.0,4.0) {$s$};
\node at (-4.5, -.3) {0};
\node at (-4, -.3) {2};
\node at (-3.5, -.3) {4};
\node at (-3, -.3) {6};
\node at (-2.5, -.3) {8};
\node at (-2, -.3) {10};
\node at (2, -.3) {0};
\node at (2.5, -.3) {2};
\node at (3, -.3) {4};
\node at (3.5, -.3) {6};
\node at (4, -.3) {8};
\node at (4.5, -.3) {10};
\node at (-5, 0) {0};
\node at (-5, 1) {1};
\node at (-5, 2) {2};
\node at (-5, 3) {3};
\node at (1.5, 0) {0};
\node at (1.5, 1) {1};
\node at (1.5, 2) {2};
\node at (1.5, 3) {3};
\node at (-2.7,1) {$h_0$};
\node at (-1.7,1.7) {$\Pi_0h_0$};
\draw[->, thick] (2,0) -- (5,0);
\draw[->, thick] (2,0) -- (2,3.5);
\draw[fill] (-4,0) circle [radius=0.05];
\draw[fill] (-3,1) circle [radius=0.05];
\draw[fill] (-2,2) circle [radius=0.05];
\draw[fill, red] (2.5,0) circle [radius=0.05];
\draw[fill, red] (3.5,1) circle [radius=0.05];
\draw[fill, red] (4.5,2) circle [radius=0.05];
\draw[->>, thick, green] (-3.95,.1) to [out=20,in=160](2.4,.1);
\draw[->>, thick, green] (-2.95,1.1) to [out=20,in=160](3.4,1.1);
\draw[->>, thick, green] (-1.85,2.1) to [out=20,in=160](4.4,2.1);
\end{tikzpicture}  
\caption{Comparing the $3$-local Adams spectral sequences
  for computing $\pi^s_\bullet(P_h)$ and $ku_\bullet(P_h)$ for
  $h=3k$, $\gcd(3,k)=1$. Red dots indicate the contribution
  from $\Ext^{s,2+t}_{\cC_*}(\bF_3,\bF_3)$.  Other contributions
  on the right are omitted.}
\label{fig:ASS2}
\end{figure}
\end{proof}  

\section{The nonsimply connected cases}
\label{sec:nonsimplyconn}

Similar techniques can also be used to compute twisted $K$-theory
for the non-simply connected simple rank-$2$ groups.
There are two of these, $\PSU(3)$ with fundamental group $\bZ/3$
and $\PSp(2)\cong \SO(5)$ with fundamental group $\bZ/2$.  The case of
$\PSU(3)$ was studied in \cite[Theorem 19 and Remark 20]{MathaiRos},
so we consider here the case of $\PSp(2)$.  Note first of all
that the covering map $\Sp(2)\xrightarrow{\pi} \PSp(2)$
induces an isomorphism on
$H^3$ by \cite[Theorem 1]{MathaiRos}, and that $\PSp(2)$ fits into
a fibration
\begin{equation}
  S^3 = \Sp(1) \to \PSp(2)\to \bR\bP^7,
\label{eq:PsPfib}
\end{equation}
which replaces the fibration $\Sp(1) \to \Sp(2)\to S^7$ used in the proof
of Theorem \ref{thm:Sp2}.  We have transfer and push-forward
maps
\[
\begin{aligned}
&\pi^*\co K_\bullet(\PSp(2),h)\to K_\bullet(\Sp(2),h)\text{ and }\\
  &\pi_*\co K_\bullet(\Sp(2),h)\to K_\bullet(\PSp(2),h),
\end{aligned}
\]
and $\pi_*\circ \pi^*$ is multiplication by $2$.  Since
$K_\bullet(\Sp(2),h)$ is cyclic in both even and odd degree,
this implies that when we localize at
an odd prime $p$, $K_\bullet(\PSp(2),h)_{(p)}\cong  K_\bullet(\Sp(2),h)_{(p)}$.
If $p=3$, this is a cyclic group of order $3^{\max(0,\nu_3(h)-1)}$,
and if $p\ge 5$, this is a cyclic group of order $p^{\nu_p(h)}$.
The only issue is therefore what happens with $2$-primary torsion.
Recall from Theorem \ref{thm:Sp2} that $K_\bullet(\Sp(2),h)_{(2)}$ is
a cyclic group of order $2^{\max(0,\nu_2(h)-1)}$.  We have by Theorem
\ref{thm:SegalSS} from \eqref{eq:PsPfib} a Segal spectral sequence
\begin{equation}
H_p(\bR\bP^7, K_q(S^3, h)) \Rightarrow K_\bullet(\PSp(2),h).
\label{eq:PSp2}
\end{equation}
If $h$ is odd, this gives $0$ after localizing at $2$. So assume
that $h=2k$ with $k$ odd. After localizing at $2$, the
left side of \eqref{eq:PSp2} becomes $H_p(\bR\bP^7, \bF_2)$ for
$q$ even, $0$ for $q$ odd.  The transfer argument shows that
multiplication by $2$ on $K_\bullet(\PSp(2),h)_{(2)}$ factors through
$K_\bullet(\Sp(2),h)_{(2)}=0$, so all $2$-primary torsion is of order $2$.

Now if $h$ is even, it is $0$ mod $2$, so we have natural maps
\[
\begin{aligned}
  K_0(S^3, h)&\xrightarrow{\text{reduce mod }2}
  K_0(S^3, h; \bF_2)\cong K_0(S^3; \bF_2), \text{ and }\\
  K_\bullet(\PSp(2),h) &\xrightarrow{\text{reduce mod }2}
  K_\bullet(\PSp(2),h; \bF_2) \cong
K_\bullet(\PSp(2); \bF_2),
\end{aligned}
\]
the first of which is an isomorphism.  So we get a map of spectral
sequences
\begin{equation}
\xymatrix{
  H_p(\bR\bP^7, K_q(S^3, h))_{(2)} \ar@{=>}[r] \ar@{^{(}->}[d]
  & K_\bullet(\PSp(2),h)_{(2)} \ar[d]\\
  H_p(\bR\bP^7, K_q(S^3; \bF_2)) \ar@{=>}[r]
  & K_\bullet(\PSp(2); \bF_2).}
\label{eq:PSp22}
\end{equation}

The $K$-theory of $\PSp(2)\cong \SO(5)$ was computed in
\cite[Satz 5.15]{MR0292104}; as an abelian group it is
$\bZ^2 \oplus \bZ/4$ in both even and odd degree.
Hence in the Segal spectral sequence $H_p(\bR\bP^7, K_q(S^3))
\Rightarrow K_\bullet(\PSp(2))$, which has a $\bZ/2$ in $E^2$
in bidegrees $(2j-1,k)$, $j=1,2,3$, there is room for only one
differential. In fact, from the description of the torsion
in $K^\bullet$ in \cite{MR0292104}, the torsion in $K^0$ is generated by the
pull-back of the generator of $\widetilde K^0(\bR\bP^7)$,
and the generator of the torsion in $K^1$ is generated by the
product of this class with an odd generator $\lambda_1$ of a torsion-free
exterior algebra, which is precisely the canonical representation
$\PSp(2)\cong \SO(5)\to U(5)$ viewed as a class in $K^1$.
So this determines the differentials in
the Segal spectral sequence in $K$-cohomology; there must be differentials
killing off $H^6(\bR\bP^7, K^0(S^3))$ and
$H^6(\bR\bP^7, K^1(S^3))$. From the
universal coefficient theorem, $K_\bullet(\PSp(2); \bF_2)\cong \bF_2^4$
in both even and odd degree.  (We get a group of rank $3$ from reducing
the integral $K$-homology mod $2$, and pick up another $\bF_2$
from the $\Tor$ term.)  If we compare the bottom spectral 
sequence in \eqref{eq:PSp22} with the one for integral $K$-homology
and with the one for twisted $K$-homology of $\SO(4)$
(in which there are no differentials at all)
we see that the only non-zero differentials are
$d^2\co E^2_{p+2, q}\to E^2_{p, q+1}$ with $p=4$ or $5$.  Now go back to
the commuting diagram \eqref{eq:PSp22}.  There cannot be a non-zero
differential in the upper spectral sequence, since it would imply
existence of a forbidden differential in the lower sequence.
So the spectral sequence for $K_\bullet(\PSp(2),h)_{(2)}$ collapses,
and since all torsion is of order $2$, we conclude that the
$2$-primary torsion in $K_\bullet(\PSp(2),h)_{(2)}$ is $(\bZ/2)^4$
in both even and odd degree.  Putting everything together, we see
that we have proved the following:
\begin{theorem}
\label{thm:SO5}  
Suppose that $h$ is either odd or $2$ mod $4$.  Then
$K_\bullet(\PSp(2)$, $h)$ is finite, and is the same in both even and
odd degree.  The odd torsion in $K_\bullet(\PSp(2), h)$ is cyclic
of order $h_{\text{odd}}/\gcd(h,3)$.  The $2$-primary torsion vanishes if
$h$ is odd, and if $h$ is $2$ mod $4$, it is $(\bZ/2)^4$ in each degree.
\end{theorem}
Cases where $h$ is divisible by a higher power of $2$ can be handled
similarly, though the results are more complicated.

\bibliographystyle{amsplain}
\bibliography{twistedK}
\end{document}